\documentclass[oneside]{amsart}

\usepackage[letterpaper,body={14.6cm,22.5cm}, mag=1000]{geometry}
\usepackage{amssymb}
\usepackage{amsthm}
\usepackage{amscd}
\usepackage{lineno}

\numberwithin{equation}{section}
\theoremstyle{plain}
\newtheorem{thm}{Theorem}[section]
\newtheorem{cor}[thm]{Corollary}
\newtheorem{lemma}[thm]{Lemma}

\newtheorem{prop}[thm]{Proposition}

\newtheorem*{thma}{Theorem A}
\newtheorem*{thmb}{Theorem B}
\newtheorem*{thmc}{Theorem C}
\newtheorem*{thmd}{Theorem D}

\theoremstyle{definition}

\newtheorem{remark}[thm]{Remark}

\newcommand{\dlabel}[1]{\ifmmode \text{\ttfamily \upshape [#1] } \else
{\ttfamily \upshape [#1] }\fi \label{#1} }
\newcommand{\A}{\operatorname{A} }
\newcommand{\B}{\operatorname{B} }
\newcommand{\C}{\operatorname{C} }

\newcommand{\Z}{\operatorname{Z} }

\newcommand{\gen}[1]{\left < #1 \right >}
\newcommand{\Aut}{\operatorname{Aut} }
\newcommand{\Cb}{\operatorname{Cb} }

\newcommand{\Hom}{\operatorname{Hom} }
\newcommand{\Inn}{\operatorname{Inn} }

\newcommand{\Autcent}{\operatorname{Autcent} }

\newcommand{\onto}{\twoheadrightarrow}

\newcommand{\annd}{\quad \text{ and } \quad}

\begin{document}
\baselineskip 15pt

\title{Class-preserving automorphisms of finite $p$-groups II }

\author{Manoj K.~Yadav}

\address{School of Mathematics, Harish-Chandra Research Institute \\
Chhatnag Road, Jhunsi, Allahabad - 211 019, INDIA}

\email{myadav@hri.res.in}


\begin{abstract}
 Let $G$ be a finite group minimally generated by $d(G)$ elements and  $\Aut_c(G)$ denote the group of all (conjugacy) 
class-preserving automorphisms of $G$. Continuing our work [Class preserving automorphisms of finite $p$-groups, 
J. London Math. Soc.  \textbf{75(3)} (2007), 755-772],  we study finite $p$-groups $G$ such that $|\Aut_c(G)| = |\gamma_2(G)|^{d(G)}$, where $\gamma_2(G)$ denotes the commutator subgroup of $G$.  If $G$ is such a $p$-group of class $2$, then we show that $d(G)$ is even, $2d(\gamma_2(G)) \le d(G)$ and  $G/\Z(G)$ is  homocyclic. When the nilpotency class of $G$ is larger than $2$, we obtain the following (surprising) results: (i) $d(G) = 2$. (ii)   If  $|\gamma_2(G)/\gamma_3(G)| > 2$,  then  $|\Aut_c(G)| = |\gamma_2(G)|^{d(G)}$ if and only if $G$ is a $2$-generator group with  cyclic commutator subgroup, where $\gamma_3(G)$ denotes the third term in the lower central series of $G$. (iii) If  $|\gamma_2(G)/\gamma_3(G)| = 2$, then $|\Aut_c(G)| = |\gamma_2(G)|^{d(G)}$ if and only if $G$ is a $2$-generator $2$-group of nilpotency class $3$ with elementary abelian commutator subgroup of order at most $8$.
As an application, we  classify   finite nilpotent groups $G$ such that the central quotient $G/\Z(G)$ of $G$ by it's center $\Z(G)$ is of the largest possible order. 
For proving these results, we introduce a generalization of  Camina  groups and obtain some interesting results.  We use Lie theoretic techniques and computer algebra system `Magma' as tools.
\end{abstract}
\maketitle

\section{Introduction}

This paper is devoted to the study of finite $p$-groups admitting maximum number of class-preserving automorphisms. An automorphism $\alpha$ of a group $G$ is called \emph{class-preserving} if 
$\alpha(x) \in x^G$ for all $x \in G$,  where $x^G$ denotes the conjugacy class of $x$ in $G$. The set of all class-preserving
automorphisms of $G$,  denoted by $\Aut_{c}(G)$, forms a normal subgroup of the group of all automorphisms of $G$, and contains $\Inn(G)$, the group of all inner automorphisms of $G$.  

For a finite group $G$ minimally generated by $d$ elements $x_1, x_2, \ldots, x_d$, it follows that
\begin{equation}
\label{bineq} |\Aut_{c}(G)| \leq \prod_{i=1}^{d} |x_{i}^G|,
\end{equation}
since there are no more choices for the generators to go under any class-preserving automorphism. Notice that \eqref{bineq} holds true for any minimal generating set $\{x_{1}, x_2, \ldots, x_{d}\}$ for $G$. Since $|x^G| = |[x, G]| \le |\gamma_2(G)|$, from  \eqref{bineq} we get
\begin{equation}
\label{lineq}  |\Aut_{c}(G)| \leq  |\gamma_2(G)|^d,
\end{equation}
where $[x, G]$ denotes the set $\{[x, g] \mid g \in G\}$ and $\gamma_2(G)$ denotes the commutator subgroup of $G$. We say that a finite group $G$, minimally generated by $d$ elements, satisfies \emph{Hypothesis A} if equality holds for it in \eqref{lineq}. Most obvious examples of groups satisfying Hypothesis A are abelian groups, and little less obvious ones being finite extraspecial $p$-groups. Notice that none of these classes of groups admit any class preserving outer automorphism.

An interesting class of  groups $G$ satisfying Hypothesis A was constructed by  Burnside \cite{wB13} (in 1913)  while answering his own question \cite[page 463]{wB55} about the existence of a finite group admitting a non-inner class-preserving automorphism. This group is of order  $p^6$ and is isomorphic to a group consisting of all $3 \times 3$ unitriangular matrices over the field $\mathbb F_{p^2}$ of $p^2$ elements, where $p$ is an odd prime. For this group $G$,  $\Inn(G) < \Aut_{c}(G)$ and $\Aut_{c}(G)$ is an elementary abelian $p$-group of order $p^8$. Notice that $G$ is minimally generated by $4$ elements and $|\gamma_2(G)| = p^2$. Thus it follows that $|\Aut_c(G)| = p^8 = |\gamma_2(G)|^4$, and therefore  equality holds in \eqref{lineq} for  $G$.

Interestingly, a generalization of the group constructed by Burnside also  enjoys this property.  
It follows from \cite[Theorem B]{BVY} that the group $G$ consisting  of all $3 \times 3$ unitriangular  matrices over a finite field $\mathbb F_{p^m}$ of $p^m$ elements satisfies Hypothesis A, where $m \ge 2$ and $p$ is an odd prime.   
A wider class of groups $G$ satisfying Hypothesis A is the class of finite Camina $p$-groups of nilpotency class $2$. A non-abelian finite group $G$ is called  \emph{Camina group} if $x^G = x\gamma_2(G)$ (or equivalently $[x, G] = \gamma_2(G)$) for all $x \in G - \gamma_2(G)$ (concept initiated by Alan Camina \cite{aC78}). So, coming back to our discussion, let $G$ be a finite Camina $p$-group of nilpotency class $2$. Then  it follows  from \cite[Theorem 5.2]{mY07} that equality holds in \eqref{lineq}.  Examples of groups of larger nilpotency class and  satisfying Hypothesis A are given below. A (bit wild) natural problem which arises here \cite[Problem 6.7]{mY11} is the following: 
\vspace{.1in}

\noindent{\bf Problem.}  Classify all finite $p$-groups $G$ satisfying Hypothesis A.

\vspace{.1in}

Since equality holds in  \eqref{lineq} for all finite abelian groups, we only consider non-abelian ones. In \cite{mY07} we considered a special case of this problem and classified, upto isoclinism (see Section 3 for the definition), all finite $p$-groups $G$ such that
\begin{equation*}
\label{gbineq} |\Aut_{c}(G)| =
   \begin{cases}
     p^{\frac{(n^{2}-4)}{4}},  &\text{if $n$ is even;}\\
     p^{\frac{(n^{2}-1)}{4}}, &\text{if $n$ is odd,}
   \end{cases}
\end{equation*}
where $n = log_p|G|$.

 In this paper, continuing our work  of \cite{mY07}, we make  a substantial progress on this problem.  Other motivation of this study is to provide a classification of finite nilpotent groups $G$ having  central quotient $G/\Z(G)$ of maximum possible order. It is done in Section 11, where some historical remarks are also made on the relationship between the orders of $G/\Z(G)$ and $\gamma_2(G)$ for an arbitrary group $G$.

 Let $G$ be a finite $p$-group minimally generated by $d$ elements which satisfies Hypothesis A. Let $\{x_1, x_2, \ldots, x_d\}$ be any minimal generating set for $G$. Then it is easy to see the following two statements: (i)   $[x_i, G] = \gamma_2(G)$ for all $x_i$, and therefore equality holds in \eqref{bineq}; (ii) For any element $x \in G - \Phi(G)$, $[x, G] = \gamma_2(G)$, where $\Phi(G)$ denotes the Frattini subgroup of $G$.  By statement (i) it follows that  equality holds in \eqref{bineq} for any minimal generating set $\{x_1, x_2, \ldots, x_d\}$ of a finite $p$-group $G$ satisfying Hypothesis A. Interestingly,  the converse of this statement also holds true. We record it in the following theorem, which we prove in Section 4. 

\begin{thma}
Let $G$ be a finite $p$-group. Then equality holds in \eqref{bineq} for all minimal generating sets $\{x_{1}, \ldots, x_{d}\}$ of $G$ if and only if equality holds for $G$ in \eqref{lineq}.
\end{thma}

The condition `$[x, G] = \gamma_2(G)$ for all $x \in G - \Phi(G)$' (noticed in the statement (ii) above) looks like the one in the definition of a Camina group. Notice that  all Camina groups enjoy this property, but the converse is not always true as  shown by metacyclic $p$-groups, where $p$ is an odd prime.  For this reason,  a non-abelian finite group $G$ will be called \emph{Camina-type} if  $[x, G] = \gamma_2(G)$  (or equivalently $x^G = x\gamma_2(G)$) for all $x \in G - \Phi(G)$. Notice that a Camina-type group $G$ is  a Camina group if and only if $\gamma_2(G) = \Phi(G)$ (Since the nilpotency class of  finite Camina $p$-groups is bounded above by $3$ \cite{DS96}, it follows from \cite[Corollary 2.3]{iM81} that $\gamma_2(G) = \Phi(G)$ for all finite Camina $p$-groups $G$).

Notice (statement (ii) above) that any finite $p$-group $G$ satisfying Hypothesis A is a Camina-type group. But when the nilpotency class of $G$ is $2$, then the converse also holds true (see  Proposition \ref{prop1b}). Thus in the case of class $2$, the class of Camina-type finite $p$-groups coincides with the class of finite $p$-groups satisfying Hypothesis A. Since all finite Camina $p$-groups of nilpotency class $2$ satisfy Hypothesis A and  classification of such groups is not known (to the best of our knowledge), it seems that a complete classification of finite $p$-groups of class $2$ satisfying Hypothesis A is a difficult task. But we have been able to obtain some interesting structural information (similar to Camina groups) for these groups, shown in the following result proved in Section 5. By $d(G)$ we denote the number of elements in a minimal generating set of a finite group $G$.

\begin{thmb}
 Let $G$ be a  Camina-type finite $p$-group of nilpotency class $2$. Then the following statements hold true:

{\rm (i)} $d(G)$ is even and $d(G) \ge 2d(\gamma_2(G))$.

{\rm (ii)} $G/\Z(G)$ is homocyclic.
 
{\rm (iii)} If $\gamma_2(G)$ is cyclic, then $G$ is isoclinic to a (full) central product of $2$-generator groups of nilpotency class $2$.
\end{thmb}

Our next result is a record of the information obtained about  Camina-type finite $p$-groups of nilpotency class larger than $2$.

\begin{thmc}
Let $G$ be a  Camina-type finite $p$-group having  nilpotency class at least $3$. Then the following statements hold true:

{\rm (i)} $d(G) = 2d(\gamma_2(G)/\gamma_3(G))$ is even, where $\gamma_3(G)$ denotes the third term in the lower central series of $G$.

{\rm (ii)} If $\gamma_2(G)$ is cyclic, then $d(G) = 2$. 

{\rm (iii)} If $d(G) =2$ and $|\gamma_2(G)/\gamma_3(G)| > 2$, then $\gamma_2(G)$ is cyclic and $G$ satisfies Hypothesis A.
\end{thmc}

We are now ready to provide a classification of finite $p$-groups satisfying Hypothesis A and having nilpotency class at least $3$.  We actually encounter some nice surprises in this case. For example, such groups can be generated by two elements. Before proceeding further, we take a little diversion to demonstrate examples of such groups (as promised above).   Consider the metacylic groups
\begin{equation}\label{Intgrp1}
K := \gen{x, y \mid x^{p^{r+t}} = 1, y^{p^r} = x^{p^{r+s}}, [x, y] = x^{p^t}},
\end{equation}
where $1 \le t < r$ and $0 \le s \le t$ ($t \ge 2$ if $p = 2$) are non-negative integers. Notice that the nilpotency class of $K$ is at least $3$. Since $K$ is generated by $2$ elements, it follows from \eqref{lineq} that $|\Aut_c(K)| \le |\gamma_2(K)|^2 = p^{2r}$. It is not so difficult  to  see that $|\Inn(K)| = |K/\Z(K)| = p^{2r}$. Since $\Inn(K) \le \Aut_c(K)$, it follows that $|\Aut_c(K)| = |\Inn(K)| = |\gamma_2(K)|^2 = |\gamma_2(K)|^{d(K)}$ (That $\Aut_c(G) = \Inn(G)$, is, in fact, true for all finite metacylic $p$-groups). Thus $K$ satisfies Hypothesis A.  Furthermore, if $H$ is any $2$-generator group isoclinic to $K$, then it follows that $H$ satisfies Hypothesis A. 

After this much diversion, we come back to the main stream and provide the desired classification in the following result, which is the main theorem of this paper. 

\begin{thmd}
Let $G$ be a finite $p$-group of nilpotency class at least $3$.  Then the following statements hold true:

{\rm (i)} If $G$ satisfies Hypothesis A, then $d(G) = 2$. 

{\rm (ii)} Let  $|\gamma_2(G)/\gamma_3(G)| > 2$. Then $G$  satisfies Hypothesis A if and only if $G$ is a $2$-generator group with cyclic commutator subgroup.  Moreover, $G$ is isomorphic to some group defined in \eqref{2ggrp} and  is isoclinic to the group $K$ defined in \eqref{Intgrp1} for suitable parameters. 

{\rm (iii)} Let $|\gamma_2(G)/\gamma_3(G)| = 2$. Then $G$  satisfies Hypothesis A if and only if $G$ is a $2$-generator $2$-group  of nilpotency class $3$ with elementary abelian $\gamma_2(G)$ of order at most $8$.
\end{thmd}

A detailed analysis of finite $2$-generator $2$-groups $G$ satisfying Hypothesis A and with $|\gamma_2(G)/$ $\gamma_3(G)| = 2$ will be presented in \cite{NY13}.

Organization of the paper is as follows. Notations are set in Section 2, and some preliminary results are given in Section 3. Section 4 is devoted to proving Theorem A and some key results. Camina-type $p$-groups of nilpotency class $2$ are discussed in Section 5, where Theorem B is proved. Section 6 deals with Camina-type $p$-groups with cyclic commutator subgroup. $2$-generator Camina-type $p$-groups are studied in Section 7. $2$-generator $2$-groups satisfying Hypothesis A are discussed in Section 8. Macdonald's arguments \cite{iM81}  are revisited in Lie theoretic setting in Section 9, which are then used to prove Theorems C and D in Section 10. 
As an application of Theorems B and D, a characterisation of finite nilpotent groups $G$ such that $|G/\Z(G)| = |\gamma_2(G)|^{d(G)}$ is given in Section 11. 


\section{Notation}

Our notation for objects associated with a finite multiplicative group $G$ is mostly standard.  We use $1$ to denote both the identity element of  $G$ and the trivial subgroup $\{ 1 \}$ of $G$. By $\Aut(G)$, $\Aut_{c}(G)$ and $\Inn(G)$, we denote the group of all automorphisms, the group of conjugacy  class 
preserving automorphisms and the group of inner automorphisms of $G$
respectively. The abelian group of all homomorphisms from an abelian group 
$H$ to an abelian group $K$ is denoted by $\Hom(H,K)$.

We write $\gen x$ for the cyclic subgroup of $G$ generated by a given element 
$x \in G$. To say that some $H$ is a subset or a subgroup of $G$ we write 
$H \subseteq G$ or $H \leq G$ respectively. To indicate, in addition, that 
$H$ is properly contained in $G$, we write $H \subset G$, $H < G$ 
 respectively.  For a subgroup $H$ of $G$, by $H^{p^k}$ we denote the subgroup of $H$ generated by the set $\{h^{p^k} \mid h \in H\}$, where $p$ is a prime and $k$ is a positive integer.  If $x,y \in G$, then $x^y$ denotes the conjugate element 
$y^{-1}xy \in G$ and $[x,y] = [x,y]_G$ denotes the commutator 
$x^{-1}y^{-1}xy = x^{-1}x^y \in G$. If $x \in G$, then $x^G$ 
denotes the $G$-conjugacy class of all $x^g$, for $g \in G$, and $[x,G]$ 
denotes the set  $\{[x, g] \mid g \in G\}$ (and not the subgroup generated by these commutators). 
If $[x,G] \subseteq \Z(G)$, then $[x,G]$ becomes a subgroup of $G$. 
Since $x^g = x[x,g]$, for all  $g \in G$, we have $x^G = x[x,G]$. Thus $|x^G| = |[x, G]|$.

Exponent of a subgroup $H$ of $G$ is denoted by $exp(H)$. For a subgroup $H$ of $G$, $\C_{G}(H)$ denotes the centralizer of $H$ in $G$ and for an element $x \in G$, $\C_{G}(x)$ denotes the centralizer of $x$ in $G$. For  a finitely generated group $G$, by $d(G)$ we denote the number of elements in any minimal generating set of $G$. 
By $\Z(G)$, $\gamma_2(G)$ and $\Phi(G)$, we  denote the center, the commutator subgroup and the Frattini subgroup of $G$ respectively. 

We write the subgroups in the lower central series of $G$ as $\gamma_n(G)$, 
where $n$ runs over all strictly positive integers. They are defined 
inductively by
\begin{align*} 
\gamma_1(G) &= G \annd \\
\gamma_{n+1}(G) &= [\gamma_n(G), G] 
\end{align*}
for any integer $n \ge 1$. Note that $\gamma_2(G)$ is the commutator subgroup 
$[G,G]$ of $G$. 
Let $x_1, x_2,\ldots,x_k$ be $k$ elements of $G$, where $k \ge 2$. 
The commutator of  $x_1$ and $x_2$ has been defined to be 
$[x_1,x_2] = x_1^{-1}x_2^{-1}x_1x_2$.
Now we define a higher commutator of $x_1, x_2,\ldots,x_k$ inductively as 
\[[x_1,x_2,\ldots,x_k] = [[x_{1},\ldots,x_{k-1}],x_k].\]
If we are given a group $G$ with a minimal generating set $X$, then  $[x_1, x_2, \ldots, x_k]$ is said to be a commutator of weight $k$ provided all $x_i \in X$, $1 \le i \le k$.

\section{Preliminary results}

 We start with the following concept of isoclinism of groups,  introduced by P. Hall \cite{pH40}.

Let $X$ be a  group and $\bar{X} = X/\Z(X)$. 
Then commutation in $X$ gives a well defined map
$a_{X} : \bar{X} \times \bar{X} \mapsto \gamma_{2}(X)$ such that
$a_{X}(x\Z(X), y\Z(X)) = [x,y]$ for $(x,y) \in X \times X$.
Two  groups $G$ and $H$ are called \emph{isoclinic} if 
there exists an  isomorphism $\phi$ of the factor group
$\bar G = G/\Z(G)$ onto $\bar{H} = H/\Z(H)$, and an isomorphism $\theta$ of
the subgroup $\gamma_{2}(G)$ onto  $\gamma_{2}(H)$
such that the following diagram is commutative
\[
 \begin{CD}
   \bar G \times \bar G  @>a_G>> \gamma_{2}(G)\\
   @V{\phi\times\phi}VV        @VV{\theta}V\\
   \bar H \times \bar H @>a_H>> \gamma_{2}(H).
  \end{CD}
\]
The resulting pair $(\phi, \theta)$ is called an \emph{isoclinism} of $G$ 
onto $H$. Notice that isoclinism is an equivalence relation among  groups.  

Let $G$ be a finite $p$-group. Then it follows from \cite{pH40} that there exists a finite $p$-group $H$ in the isoclinism family of $G$ such that 
$\Z(H) \le \gamma_2(H)$. Such a group $H$ is called a \emph{stem group} in the isoclinism family of $G$.  

The following theorem shows that the group of class-preserving automorphisms is independent of the choice of a group in a given isoclinism family of groups.

\begin{thm}[\cite{mY08}, Theorem 4.1]\label{mY08}
Let $G$ and $H$ be two finite non-abelian isoclinic groups. Then
$\Aut_c(G) \cong \Aut_c(H)$.
\end{thm}

\begin{lemma}\label{prelemma}
Let $G$ be a finite $p$-group satisfying Hypothesis A, and $H$ be a stem group in the isoclinism family of $G$. Then $d(G) = d(H)$ and $H$ satisfies Hypothesis A.
\end{lemma}
\begin{proof}
Since $G$ satisfies Hypothesis A, it follows that $\Z(G) \le \Phi(G)$.  Also, $H$ being a stem group,  $\Z(H) \le \Phi(H)$. Since $G/\Z(G) \cong H/\Z(H)$, $d(G) = d(G/\Z(G))$ and $d(H) = d(H/\Z(H))$, we get $d(G) = d(H)$. Now by Theorem \ref{mY08} we have $\Aut_c(G) \cong \Aut_c(H)$. As $G$ satisfies Hypothesis A, $|\Aut_c(G)| = |\gamma_2(G)|^{d(G)}$. Thus $|\Aut_c(H)| = |\gamma_2(H)|^{d(H)}$, since $\gamma_2(G) \cong \gamma_2(H)$. This proves that $H$ satisfies Hypothesis A and the proof of the lemma is complete. \hfill $\Box$

\end{proof}

The following lemma is a kind of variant of the preceding lemma with a weaker hypothesis.

\begin{lemma}\label{prelemma1}
Let $G$ be a Camina-type  finite $p$-group, and  $H$ be a stem group in the isoclinism family of $G$. Then $d(G) = d(H)$ and $H$ is also Camina-type. 
\end{lemma}
\begin{proof}
Let $G$ be a group as in the statement. Then  it follows that $\Z(G) \le \Phi(G)$.  Also, $H$ being a stem group,  $\Z(H) \le \Phi(H)$. Since $G/\Z(G) \cong H/\Z(H)$, $d(G) = d(G/\Z(G))$ and $d(H) = d(H/\Z(H))$, we get $d(G) = d(H)$.  Let $(\phi, \theta)$ be an isoclinism of $G$ onto $H$.  Let $h \in H - \Phi(H)$ be an arbitrary element. Then there exists an element $g \in G -\Phi(G)$ such that $\phi(g\Z(G)) = h\Z(H)$. Thus by the given hypothesis $[g, G] = \gamma_2(G)$. Hence $\theta ([g, G]) = \theta (\gamma_2(G)) = \gamma_2(H)$. By the definition of isoclinism,  for any $x \in G$, $\theta([g, x]) = [h, y]$, where $\phi(x\Z(G)) = y\Z(H)$. Now $\theta([g, G] ) = \{\theta([g, x]) \mid x \in G\}$. Thus $\gamma_2(H) = \theta ([g, G]) = \{[h, y] \mid y \in H\} = [h, H]$, since $[h, y] = 1$ if $y \in \Z(H)$. This shows that $H$ is a Camina-type group and the proof is complete.  \hfill $\Box$

\end{proof}

The following result follows from the definition of Camina-type groups.
\begin{lemma}\label{prelemma2}
Let $G$ be a finite nilpotent Camina-type group. Then $G$ is a $p$-group for some prime $p$.
\end{lemma}

The following result is due to Brady, Bryce and Cossey.
 
\begin{thm}[\cite{BBC}, Theorem 2.1]\label{BBC}
Let $G$ be a finite $p$-group of nilpotency class $2$ with cyclic center.  Then $G$ is a central product either of  two generator subgroups with cyclic center  or   two generator subgroups with cyclic center and a cyclic subgroup.
\end{thm}

The following result follows from Theorem 3.2 of Macdonald's paper \cite{iM81}.
\begin{thm}\label{iM81}
Let $G$ be a finite Camina $p$-group of class $2$ such that $d(G) = m$ and $d(\gamma_2(G)) = n$. Then  $m$ is even and $2n \le m$.
\end{thm}

\begin{thm}[\cite{mY07}, Theorem 5.1]\label{mY07}
Let $G$ be a finite $p$-group  such that $|G| = p^n$ and $|\gamma_2(G)| = p^m$. If $|\Aut_c(G)| = p^{m(n-m)}$, then $G$ is a Camina group of class $2$.
\end{thm}

An automorphism $\phi$ of a group $G$ is called  \emph{central} if 
$g^{-1}\phi(g) \in \Z(G)$ for all  $g \in G$. The set of all central 
automorphisms of $G$, denoted by $\Autcent(G)$, is a normal subgroup 
of $\Aut(G)$. Notice that $\Autcent(G) = \C_{\Aut(G)}(\Inn(G))$. 

Let $\alpha \in \Autcent(G)$. Then the map $f_{\alpha}$ from $G$ into $\Z(G)$
defined by $f_{\alpha}(x) = x^{-1}\alpha(x)$ is a homomorphism which sends
$\gamma_2(G)$ to $1$. Thus $f_{\alpha}$ induces a homomorphism from
$G/\gamma_2(G)$ into $\Z(G)$. So we get a one-to-one map $\alpha \rightarrow
f_{\alpha}$ from $\Autcent(G)$ into $\Hom(G/\gamma_2(G), \Z(G))$. Conversely, 
if $f \in \Hom(G/\gamma_2(G), \Z(G))$, then $\alpha_f$ such that
$\alpha_f(x) = xf({\bar x})$ defines an endomorphism of $G$, where 
${\bar x} = x \gamma_2(G)$ . But this, in general, may not be an automorphism of $G$. More precisely, $\alpha_f$ fails to be an automorphism of $G$ when $G$ admits a non-trivial abelian direct factor. 

A group $G$ is called \emph{purely non-abelian} if it does not have a  non-trivial abelian direct factor.

The following theorem of Adney and Yen \cite{AY65} shows that if $G$ is a purely non-abelian finite group, then the mapping   $\alpha \rightarrow
f_{\alpha}$ from $\Autcent(G)$ into $\Hom(G/\gamma_2(G), \Z(G))$, defined above,  is also onto.

\begin{thm}[\cite{AY65}, Theorem 1]\label{AY65}
Let $G$ be a purely non-abelian finite group. Then the correspondence $\alpha
\rightarrow f_{\alpha}$ defined above is a one-to-one mapping of $\Autcent(G)$
onto $\Hom(G/ \gamma_2(G),$ $ \Z(G))$.
\end{thm}

The following lemma is due to  Cheng \cite[Lemma 1]{yC82}. 

\begin{lemma}\label{yC82}
Let $G = \gen{x, y}$ be a $2$-generator finite  $p$-group with cyclic commutator subgroup $\gamma_2(G) = \gen{u}$ of order $p^r$ for some positive integer $r$. Assume that either $p$ is odd, or $p = 2$ and $[u, G] \subseteq \gen{u^4}$. Then the following statements hold true:

{\rm (i)} $\gamma_2(G) = \{[x, y^k] \mid 0~ \le k < p^r\}$;

{\rm (ii)} If $\phi$ is an automorphism of $G$ which induces the identity automorphism on $G/\gamma_2(G)$, then $\phi \in \Inn(G)$. 
\end{lemma}

\section{Proof of Theorem A and key results}

We start with the proof of Theorem A.

\begin{thm}[Theorem A]\label{prop0a}
Let $G$ be a finite $p$-group. Then equality holds in \eqref{bineq} for all minimal generating sets $\{x_{1}, \ldots, x_{d}\}$ of $G$ if and only if equality holds for $G$ in \eqref{lineq}.
\end{thm}
\begin{proof}
Assume that $|\Aut_c(G)| = \Pi_{i=1}^d |x_i^G|$ for any minimal generating 
set $\{x_1, \ldots, x_d\}$ of $G$.
Let $y_1 \in G - \Phi(G)$ such that $|y_1^G| \ge |y^G|$ for all $y \in G -
\Phi(G)$. Let $H_0 = \Phi(G)$ and define $H_1 = \gen{\Phi(G), y_1}$. 
Now inductively define $H_{i} =
\gen{H_{i-1}, y_i}$ such that $|y_i^G| \ge |y^G|$ for all $y \in G - H_{i-1}$,
where $1 \le i \le d$. Notice that the elements $y_1, \ldots, y_d$ form a
minimal generating set for $G$. So by the given hypothesis 
$|\Aut_c(G)| = \Pi_{i=1}^d |y_i^G|$. This implies that for given $g_1,
\ldots, g_d \in G$, there exists an automorphism $\alpha \in \Aut_c(G)$ such
that $\alpha (y_i) = y_i^{g_i}$, where $1 \le i \le d$.

First we prove that $|\Aut_c(G)| = (|y_1^G|)^d$.
Since $y_{i} \in H_{i}-H_{i-1}$
and $y_{i-1} \in H_{i-1}$, it follows that $y_{i-1}y_i \in H_{i}-H_{i-1}$,
where $2 \le i \le d$. 
By our choice of $y_i$ we have $|(y_{i-1}y_i)^G| \le |y_i^G|$. Now given
any two elements $g_{i-1}, g_i \in G$, 
there exists an automorphism $\alpha \in \Aut_c(G)$ such that $\alpha (y_j) =
y_j$ for all $1 \le j \le i-2$ and $i+1 \le j \le d$, 
$\alpha (y_{i-1}) = y_{i-1}^{g_{i-1}}$ and $\alpha (y_i) = y_i^{g_{i}}$. So
$$y_{i-1}^{g_{i-1}}y_i^{g_i} = \alpha (y_{i-1})\alpha (y_i) = 
\alpha (y_{i-1}y_i) \in (y_{i-1}y_i)^G.$$
This simply implies that $y_{i-1}^Gy_i^G \subseteq (y_{i-1}y_i)^G$. Since
$(y_{i-1}y_i)^G$ is always contained in $y_{i-1}^Gy_i^G$, we have 
\begin{equation}
\label{eqn1} y_{i-1}^Gy_i^G = (y_{i-1}y_i)^G.
\end{equation}
 This, using the fact that 
$|(y_{i-1}y_i)^G| \le |y_i^G|$, gives $|y_{i-1}^Gy_i^G| \le |y_i^G|$. 
Since $y_i \in G - H_{i-1} \subseteq G - H_{i-2}$, the definition of
$y_{i-1}$ tells us that $|y_i^G| \le |y_{i-1}^G|$.  Hence
\[ |y_i^G| \le |y_{i-1}^G| \le |y_{i-1}^Gy_i^G| \le |y_i^G|,  \]
so that $|y_{i-1}^G| = |y_i^G|$ for each $i$ such that $2 \le i \le d$.
This proves that $|y_1^G| = |y_2^G| = \cdots = |y_d^G|$ and therefore
\begin{equation}
\label{eqn2} |\Aut_c(G)| = (|y_1^G|)^d.
\end{equation}

Let $x$ be an arbitrary element of $G-\Phi(G)$. We can extend $x$ to a 
minimal generating set $\{x_1 = x, x_2, \ldots, x_d\}$ of $G$. By the given
hypothesis $|\Aut_c(G)| = \Pi_{i=1}^d |x_i^G|$ and therefore
$\Pi_{i=1}^d |x_i^G| = (|y_1^G|)^d$. Since $|x_i^G| \le |y_1^G|$ by the 
definition of $y_1$, this implies that $|x_i^G| = |y_1^G|$,
for $1 \le i \le d$. Thus $|x^G| = |x_1^G| = |y_1^G|$. This proves that
$|x^G| = |y^G|$ for all $x, y \in G-\Phi(G)$.

Now let $x,y$ be two distinct elements of some minimal generating set for $G$.
Then it follows from \eqref{eqn1} that $x^Gy^G = (xy)^G$.  
Notice that  $x^G = x[x,G]$, $y^G = y[y,G]$ and $(xy)^G
= xy[xy,G]$. The equation $x^Gy^G = (xy)^G$ then implies that $x[x,G]y[y,G]
= xy[x,G]^y[y,G]$ is equal to $xy[xy,G]$, i.~e., that
\begin{equation}
\label{eqn3}  [x,G]^y[y,G] = [xy,G].  
\end{equation}

Since $x, y$ and $xy$ all lie in $G - \Phi(G)$, it follows that
$|x^G| = |y^G| = |(xy)^G|$, and hence that $|[x,G]| = |[y,G]| =
|[xy,G]|$. Since both $[x,G]$ and $[y,G]$ contain $1$, and $G$ is a
finite group, this and \eqref{eqn3} imply that
\begin{equation}
\label{eqn4} [x,G]^y = [y,G] = [xy,G].   
\end{equation}

 Substituting \eqref{eqn4} in \eqref{eqn3}, we obtain $[y,G][y,G] = [y,G]$.  
Since $[y,G]$ is a non-empty subset of the finite group $G$, this forces it to
be a subgroup of $G$.  Furthermore, this subgroup remains the same when
$y$ is replaced by $xy$, i. e., $[y,G]=[xy,G]$. Interchanging the role of $x$
and $y$ we get $[x,G]=[yx,G]$. Since $xy$ and $yx$ are conjugate, it follows
that $[xy, G] = [yx,G]$ and therefore $[x,G] = [y, G]$. This implies that 
$[x_1,G] = \cdots = [x_d,G] = [x,G]$ (say) for any minimal generating set
$X = \{x_1, \ldots, x_d\}$ of $G$. Now we claim that $[x,G] = [u,G]$ for every $u
\in G-\Phi(G)$. So let $u \in G-\Phi(G)$ be an arbitrary element. Then $u =
x_{i_1}^{a_1}x_{i_2}^{a_2} \cdots x_{i_t}^{a_t}$, where $x_{i_j} \in X$ and
$a_j$ is a positive integer for $1 \le  j \le t$. Notice that 
$[x_{i_j}^{a_j},G] \le [x,G]$ for all $j$ such that $1 \le  j \le t$ and 
$[x_{i_j}^{a_j},G] = [x,G]$ if $x_{i_j}^{a_j} \not\in \Phi(G)$. This implies
that $[x_{i_j}^{a_j} x_{i_k}^{a_k},G] \le [x,G]$, for $1 \le j, k \le t$. So 
$[u,G] \le [x,G]$. Since $u \in G-\Phi(G)$,  some $x_{i_j}^{a_j}$ must be 
outside $\Phi(G)$. This proves that $[u,G] = [x,G]$ and hence our claim follows.  

So we have proved that  the subgroup
$[y,G]$ is independent of the choice of $y \in G - \Phi(G)$ whenever $d
\ge 2$, i.~e., whenever $G$ is non-cyclic.
Of course $[y,G] = 1$ is
also independent of the choice of $y$ when $G$ is cyclic.  Since the
derived group $\gamma_2(G) = [G,G]$ is generated by all the $[y,G]$ for
$y \in G - \Phi(G)$, this proves $[y,G] = \gamma_2(G)$ for all $y \in
G-\Phi(G)$. So in particular $|y_1^G| = |[y_1, G]| = |\gamma_2(G)|$. Hence it
follows from \eqref{eqn2} that $|\Aut_c(G)| = |\gamma_2(G)|^d$.

Conversely suppose that $|\Aut_c(G)| = |\gamma_2(G)|^d$ and $\{x_1,
\ldots, x_d\}$ is an arbitrary minimal generating set for $G$. Then 
\[|\gamma_2(G)|^d = |\Aut_c(G)| \le \Pi_{i=1}^{d} |x_i^G| \le
|\gamma_2(G)|^d,\] 
since $|x_i^G| \le |\gamma_2(G)|$ for all $i$ between $1$ and $d$. This proves
that $|\Aut_c(G)| = \Pi_{i=1}^{d} |x_i^G|$ for every minimal generating set 
$\{x_1,\ldots, x_d\}$ of $G$. This completes the proof of the theorem.
\hfill $\Box$

\end{proof}

\begin{remark}
Theorem \ref{prop0a} also holds true if we replace Hypothesis A by the fact that $|\Cb(G)| = |\gamma_2(G)|^{d(G)}$, where $\Cb(G)$ denotes the group of all basis conjugating automorphisms of $G$. These are the automorphisms which map each element  of $G - \Phi(G)$ to its conjugate. Proof goes word to word with the proof of Theorem \ref{prop0a}.
\end{remark}



The following is a key lemma, which tells that  Hypothesis A passes to factor groups.

\begin{lemma}\label{lemma1}
Let $G$ be a finite $p$-group which satisfies Hypothesis A, and $N$ be any
non-trivial proper normal subsubgroup of $G$. Then $G/N$ satisfies Hypothesis A.
\end{lemma}
\begin{proof}
Let $S := \{x_1, x_2, \ldots, x_d\}$ be a minimal generating set for $G$, where $d$ is some positive integer. Since $G$ satisfies Hypothesis A, we have  $|\Aut_c(G)| = |\gamma_2(G)|^d$. Thus it follows that given any $v_1, v_2, \ldots, v_d \in \gamma_2(G)$, there exists an $\alpha \in \Aut_c(G)$ such that $\alpha(x_i) = x_iv_i$, where $1 \le i \le d$. Let $N$ be a normal subgroup of $G$. Let us reorder the given generating set such that $\{x_1N, x_2N, \ldots, x_tN\}$ generates $G/N$ for some positive integer $t \le d$ and $x_iN = N$ for all $i > t$. Since $\gamma_2(G/N) = \gamma_2(G)N/N$, any element of $\gamma_2(G/N)$ will be of the form $uN$ for some $u \in \gamma_2(G)$. 
To establish the proof of our lemma, we only need to show that $|\Aut_c(G/N)| = |\gamma_2(G/N)|^t$. It simply requires to show that given any $t$ elements $u_1N, u_2N, \ldots, u_tN$ in $\gamma_2(G/N)$, there exists an automorphism $\bar{\alpha} \in \Aut_c(G/N)$ such that $\bar{\alpha}(x_iN) = x_iN u_iN$ for $1 \le i \le t$. 
Let $u_1, u_2, \ldots, u_t, u_{t+1}, \ldots, u_d \in \gamma_2(G)$ such that $u_{t+1} = \cdots = u_d = 1$. Then there exists an automorphism $\alpha \in \Aut_c(G)$ such that $\alpha(x_i) = x_iu_i$, where $1 \le i \le d$. Since $N$ is normal in $G$, $\alpha$ keeps $N$ invariant and therefore induces an automorphism $\bar{\alpha}$ (say) on $G/N$. Notice that $\bar{\alpha}$ is the required automorphism of $G/N$, which completes the proof of the lemma.  \hfill $\Box$

\end{proof}

One other key lemma tells us that like Camina groups, the property of being Camina-type group also passes to its factor groups by normal  subgroups contained in the commutator subgroup.

\begin{lemma}\label{lemma2}
Let $G$ be a finite Camina-type group, and $N$ be its normal subgroup such that $N \le \gamma_2(G)$. Then $G/N$ is also  a Camina-type group.
\end{lemma}
\begin{proof}
Let $G$ and $N$ be as in the statement. Consider the natural homomorphism $\pi : G \rightarrow \bar{G} = G/N$ from $G$ onto $G/N$.  Let ${\bar x} = xN \in \bar{G} - \Phi(\bar{G})$ be an arbitrary element. Then $x \in G - \Phi(G)$, and therefore $[x, G] =  \gamma_2(G)$. Now $\pi(\gamma_2(G)) = \gamma_2(G)/N = \gamma_2(G/N)$.
Also $\pi([x, G]) = [\bar{x}, \bar{G}]$. Hence  $[\bar{x}, \bar{G}] = \gamma_2(G/N)$, and the proof is complete. \hfill $\Box$

\end{proof}

\section{Camina-type groups  of class $2$}

In this section we study  Camina-type finite $p$-groups of nilpotency class $2$ and prove Theorem B. We start with the following result.

\begin{thm}\label{thmcl2}
Let $G$ be a  Camina-type finite $p$-group of nilpotency class $2$. Then $d(G)$ is even and $2d(\gamma_2(G)) \le d(G)$.
\end{thm}
\begin{proof}
Let $G$ be a group as given in the statement. Consider the quotient group ${\bar G} = G/\Phi(\gamma_2(G))$. Notice that $d(G) = d({\bar G})$ and $d(\gamma_2(G)) = d(\gamma_2({\bar G}))$. Let $H$ be a stem group in the isoclinism family of ${\bar G}$. Then $\gamma_2(H) = \Z(H)$ is of exponent $p$ and  $H$ is a Camina-type group. Hence $H$ is  a Camina group. Now it follows from Theorem \ref{iM81} that $d(H)$ is even and $2d(\gamma_2(H)) \le d(H)$. Since $d({\bar G}) = d(H)$ (by Lemma \ref{prelemma1}) and $\gamma_2({\bar G}) \cong \gamma_2(H)$, assertion of the theorem follows. \hfill $\Box$

\end{proof}

Let $G$ be a finite nilpotent group of  class $2$. Let $\phi \in
\Aut_c(G)$. Then the map $g \mapsto g^{-1}\phi(g)$ is a homomorphism of
$G$ into $\gamma_2(G)$. This homomorphism sends $Z(G)$ to $1$. So it
induces a homomorphism $f_{\phi} \colon G/Z(G) \to \gamma_2(G)$, sending
$gZ(G)$ to $g^{-1}\phi(g)$, for any $g \in G$.  It is easily seen that
the map $\phi \mapsto f_{\phi}$ is a monomorphism of the group
$\Aut_c(G)$ into $\Hom(G/Z(G), \gamma_2(G))$.

Any $\phi \in \Aut_c(G)$ sends any $g \in G$ to some $\phi(g)
\in g^G$. Then $f_{\phi}(gZ(G)) = g^{-1}\phi(g)$ lies in $g^{-1}g^G =
[g,G]$.  Denote
\[  \{ \, f \in \Hom(G/Z(G), \gamma_2(G)) \mid f(gZ(G)) \in [g,G], \text{ for
all $g \in G$}\,\} \]
by $\Hom_c(G/Z(G), \gamma_2(G))$. Then $f_{\phi} \in \Hom_c(G/Z(G),
\gamma_2(G))$ for all $\phi \in \Aut_c(G)$. On the other hand, if $f \in
\Hom_c(G/Z(G), \gamma_2(G))$, then the map sending any $g \in G$ to
$gf(gZ(G))$ is an automorphism $\phi \in \Aut_c(G)$ such that $f_{\phi}
= f$. Thus we have

\begin{prop}\label{prop1}
 Let $G$ be a finite nilpotent group of class 2. Then the
above map $\phi \mapsto f_{\phi}$ is an isomorphism of the group
$\Aut_c(G)$ onto $\Hom_c(G/Z(G), \gamma_2(G))$.
\end{prop}

The following two lemmas are well known.
\begin{lemma}\label{lemma5}
Let $\A$, $\B$ and $\C$ be finite abelian groups. Then

{\rm (i)} $\Hom(\A \times \B, \C) \cong \Hom(\A,\C) \times \Hom(\B,\C)$;

{\rm (ii)}  $\Hom(\A,  \B \times \C) \cong \Hom(\A,\B) \times \Hom(\A,\C)$.
\end{lemma}

\begin{lemma}\label{lemma6}
Let $\C_n$ and $\C_m$ be two cyclic groups of order $n$ and $m$
respectively. Then $\Hom(\C_n, \C_m) \cong \C_d$, where $d$ is the greatest
common divisor of $n$ and $m$, and $\C_d$ is the cyclic group of order $d$.
\end{lemma} 

\begin{prop}\label{prop1a}
Let $G$ be a finite $p$-group of class $2$ satisfying Hypothesis A. Then $G/\Z(G)$ is homocyclic.
\end{prop}
\begin{proof}
Let $G/\Z(G) \cong \C_{p^{m_1}} \times \cdots \times \C_{p^{m_r}}$ for some integers   $m_1 \geq m_2 \geq \cdots \geq m_r \ge 1$. Let the exponent of $\gamma_2(G)$ be $p^e$. Then the  exponent of $G/\Z(G)$ is also $p^e$. Notice that $m_1 = m_2 = e$. We want to prove that  $m_i = e$ for each $i$ such that $1 \le i \le r$. If $r = 2$, we are done. So suppose that $r \geq 3$. Contrarily assume that  $m_k < e$ for some  $k$ such that  $3 \le k \le r$. Now 
\begin{eqnarray}\label{homeqn}
\Hom(G/\Z(G), \gamma_2(G)) &\cong& \Hom(\C_{p^{m_1}} \times \cdots \times \C_{p^{m_r}}, \gamma_2(G))\\
 &\cong&  \Hom(\C_{p^{m_1}}, \gamma_2(G)) \times \cdots \times  \Hom(\C_{p^{m_r}}, \gamma_2(G)).\nonumber
\end{eqnarray}

By considering cyclic factorization of $\gamma_2(G)$, it is easy to show that $|\Hom(\C_{p^{m_k}}, \gamma_2(G))| < |\gamma_2(G)|$. Now using \eqref{homeqn}, we have  
\[|\Aut_c(G)| = |Hom_c(G/Z(G), \gamma_2(G))| \le |Hom(G/Z(G), \gamma_2(G))| < |\gamma_2(G)|^r.\]
 Notice that $\Z(G) \le \Phi(G)$. Thus $r = d(G)$ and therefore $|\Aut_c(G)| < |\gamma_2(G)|^{d(G)}$. This contradicts the fact that $G$ satisfies Hypothesis A. Hence $m_i = e$ for each $i$ such that $1 \le i \le r$, which proves that $G/\Z(G)$ is homocyclic. \hfill $\Box$

\end{proof}

We can even prove the above proposition directly for Camina-type $p$-groups using different methods.

\begin{prop}\label{cl2prop}
Let $G$ be a Camina-type  finite $p$-group of class $2$.  Then $G/\Z(G)$ is homocyclic.
\end{prop}
\begin{proof}
Since $G$ is a  Camina-type  finite $p$-group of class $2$,  $\Z(G) \le \Phi(G)$. Let $x_1, \ldots, x_d$ be a minimal generating set for $G$ such that $G/\Z(G) = \gen{x_1\Z(G)} \times \cdots \times \gen{x_d\Z(G)}$ be a cyclic (direct)  decomposition of $G/\Z(G)$. Notice that such a minimal generating set always exists for any finite $p$-group $K$ of class $2$ with $\Z(K) \le \Phi(K)$ (\cite[Lemma 3.5(1)]{mY13}). Let the exponent of $\gamma_2(G)$ be $p^e$ for some positive integer $e$ and  the order of  $x_i$ modulo $\Z(G)$ be $p^{e_i}$ for $1 \le i \le d$. Since $p^{e_i} \le p^e = exp(G/\Z(G))$, to complete the proof, it is sufficient to show that  $p^{e_i} \ge p^e$ for each $i$. Let $u \in \gamma_2(G)$ be such that order of $u$ is $p^e$. Since $G$ is a Camina-type group, $[x_i, G] = \gamma_2(G)$ for $1 \le i \le d$. Thus, for each $i$, there exists some $g_i \in G$ such that $u = [x_i, g_i]$. Now $u^{p^{e_i}} = [x_i, g_i]^{p^{e_i}} = [x_i^{p^{e_i}}, g_i] = 1$, since $x_i^{p^{e_i}} \in \Z(G)$. This shows that $p^e \le p^{e_i}$ for each $i$ such that $1 \le i \le d$ and the proof of the lemma is complete.       \hfill $\Box$
  
\end{proof}

Next two results are about the relationship between class-preserving and central automorphisms.
\begin{lemma}\label{lemma7}
Let $G$ be a  Camina-type  finite $p$-group of class $2$,  and  $H$ be a stem group in the isoclinism family of $G$. Then $\Autcent(H) = \Aut_c(H)$.
\end{lemma}
\begin{proof}
Let $G$ and $H$ be the groups as in the statement. Since $G$ is a Camina-type $p$-group, Lemma \ref{prelemma1} tells that $H$ is also a Camina-type $p$-group.  Since $\Z(H) = \gamma_2(H)$,  $H$ is purely non-abelian. It now follows from Theorem \ref{AY65} that there is a bijection between $\Autcent(H)$ and $\Hom(H/\gamma_2(H), \Z(H))$.  Since $\Aut_c(H) \le \Autcent(H)$ for any $p$-group $H$ of nilpotency class $2$,  we only need to show that $\Autcent(H) \le \Aut_c(H)$  to complete the proof.  Let $\alpha \in \Autcent(H)$ be an arbitrary central automorphism of $H$. Then 
$f_{\alpha} \in \Hom(H/\gamma_2(H), \Z(H))$. Since $H$ is a Camina-type group, for any $x \in H - \Phi(H)$ we have   $[x, H] = \gamma_2(H) = \Z(H)$. Thus $f_{\alpha}(x\gamma_2(H))\in [x, H]$. Let $x\gamma_2(H) \in \Phi(H)/\gamma_2(H)$. Then there exists an element $y \in  H - \Phi(H)$ such that $x = y^t$ (modulo $\gamma_2(H)$) for some integer $t$. Notice that $[x, H] = [y^t, H] = [y, H]^t$, since $\gamma_2(H) = \Z(H)$. Thus 
\[f_{\alpha}(x\gamma_2(H)) = f_{\alpha}(y^t\gamma_2(H)) = f_{\alpha}(y\gamma_2(H))^t \in [y, H]^t = [x, H].\]
 Hence $f_{\alpha} \in \Hom_c(H/\Z(H), \gamma_2(H))$ and therefore $\alpha \in \Aut_c(H)$ (by Proposition \ref{prop1}). This proves that $\Autcent(H) \le \Aut_c(H)$.  \hfill $\Box$

\end{proof}

In \cite{mY13} we proved that if $\Autcent(G) = \Aut_c(G)$ for any finite $p$-group $G$, then $\Z(G) = \gamma_2(G)$. 
Using this and the above lemma, we here get

\begin{prop}
Let $G$ be a  Camina-type  finite $p$-group of class $2$. Then $\Autcent(G) = \Aut_c(G)$ if and only if $\Z(G) = \gamma_2(G)$.
\end{prop}

As we noticed in the introduction that any finite $p$-group satisfying Hypothesis A is a Camina-type group. We now prove the converse of this statement in the case of nilpotency class $2$.
\begin{prop}\label{prop1b}
Let $G$ be a  Camina-type  finite $p$-group of nilpotency class $2$. Then $G$ satisfies Hypothesis A. 
\end{prop}
\begin{proof}
Suppose that  $H$ is a stem group in the isoclinism family of $G$. Then, $H$ being Camina-type (by Lemma \ref{prelemma}), by Proposition \ref{cl2prop} and Lemma \ref{lemma7}, we have $H/\Z(H)$ is homocyclic and $\Autcent(H) = \Aut_c(H)$ respectively. Hence 
\[|\Aut_c(H)| = |\Autcent(H)| = \Hom(H/\gamma_2(H), \Z(H)) = \Hom(H/\Z(H), \gamma_2(H)) = |\gamma_2(H)|^{d(H)},\] 
since $\gamma_2(H) = \Z(H)$. 
This proves that $H$ satisfies Hypothesis A. That $G$ satisfies Hypothesis A, now can be shown on the lines of the proof of Lemma \ref{prelemma}, since $\Z(G) \le \Phi(G)$. This completes the proof of the proposition.
  \hfill $\Box$

\end{proof}

Now onwards, in this section, we mainly concentrate on characterizing finite $p$-groups $G$ of class $2$ satisfying Hypothesis A and such that $\Aut_c(G) = \Inn(G)$. We first prove 

\begin{thm}\label{thmcl2a}
Let $G$ be a finite $p$-group of class $2$ satisfying Hypothesis A. Then  $\Aut_c(G) = \Inn(G)$ if and only if $\gamma_2(G)$ is cyclic.
\end{thm}
\begin{proof}
Since $\Aut_c(G) = \Inn(G)$ for any finite $p$-group $G$ of nilpotency class $2$ with $\gamma_2(G)$ cyclic (\cite[Corollary 3.6]{mY08}), if part of the assertion follows. Now suppose that  $\Aut_c(G) = \Inn(G)$. Let $H$ be a stem group in the isoclinism family of $G$. Since $\gamma_2(G) \cong \gamma_2(H)$, it is sufficient to prove that $\gamma_2(H)$ is cyclic.  Notice that $\Aut_c(H) = \Inn(H)$. Also notice that  $H$ is a Camina-type $p$-group. Hence by Proposition \ref{cl2prop} and Lemma \ref{lemma7}, we have $H/\Z(H)$ is homocyclic of exponent equal the exponent of $\gamma_2(H)$ and $\Autcent(H) = \Aut_c(H)$ respectively.

 Suppose contrarily that $\gamma_2(G) \cong \gamma_2(H)$ is not cyclic. Then $\gamma_2(H) \cong \C_{p^e} \times K$, where $p^e$ is the exponent of $\gamma_2(H)$ and $K$ is some non-trivial subgroup of $H$ contained in $\gamma_2(H)$. Since $\gamma_2(H) = \Z(H)$,  by Theorem \ref{AY65} we have
\begin{eqnarray*}
|\Autcent(H)| &=& |\Hom(H/\Z(H), \gamma_2(H))| = |\Hom(H/\Z(H), \C_{p^e}) \times \Hom(H/\Z(H), K)|\\
& =& |H/\Z(H)| |\Hom(H/\Z(H), K)|.
\end{eqnarray*}
This gives 
\[|\Inn(H)| = |\Aut_c(H)| = |H/\Z(H)| |\Hom(H/\Z(H), K)| > |\Inn(H)|,\]
which is not possible.  Hence $\gamma_2(H)$ is cyclic and the proof of the theorem is complete. \hfill $\Box$

\end{proof}

Let $G$ be a $2$-generator finite $p$-group of class $2$. Then $\gamma_2(G)$ is cyclic and therefore $|\Aut_c(G)|$ $ = 
|\Inn(G)| = |\gamma_2(G)|^2$.  Thus $G$ satisfies Hypothesis A. Let $q = p^e$ be the order of $\gamma_2(G)$. Then $|G/\Z(G)| = q^2$. 
If we assume that $\gamma_2(G) = \Z(G)$, then $|G| = q^3$. More generally, for any positive integer $m$, let $G_1, G_2, \ldots, G_m$ be  $2$-generator finite $p$-groups such that $\gamma_2(G_i) = \Z(G_i) \cong X$ (say) is cyclic of order  $q$ for $1 \le i \le m$. Consider the central product 
\begin{equation}\label{ygroup}
Y = G_1 *_X G_2 *_X \cdots *_X G_m
\end{equation}
 of $G_1, G_2, \ldots, G_m$ amalgamated at $X$ (isomorphic to cyclic commutator subgroups $\gamma_2(G_i)$, $1 \le i \le m$).  Then $|Y| = q^{2m+1}$ and $|Y/\Z(Y)| = q^{2m} = |\gamma_2(Y)|^{d(Y)}$ as  $d(Y) = 2m$. Since $\gamma_2(Y)$ is cyclic, $\Aut_c(Y) = \Inn(Y)$ and therefore $Y$ satisfies Hypothesis A. In the following result, we show that these are the only finite $p$-groups (upto isoclinism) of class $2$ with cyclic commutator subgroup and satisfying Hypothesis A.

\begin{thm}\label{thmcl2b}
Let $G$ be a finite $p$-group of class $2$ satisfying Hypothesis A and with cyclic commutator subgroup. Then $G$  is isoclinic to the group $Y$, defined in \eqref{ygroup}, for a suitable positive integer $m$.
\end{thm}
\begin{proof}
Let $G$ be a group as in the statement and $H$ be a stem group in the isoclinism family of $G$. Then $\gamma_2(H) = \Z(H)$ is cyclic of order $q = p^e$ (say) and, by Proposition \ref{prop1a},  $H/\Z(H)$ is homocyclic of exponent $q$ and is of order $q^{2m}$ for some positive integer $m$, since $d(H)$ is even (by Theorem \ref{thmcl2}).   Since $\Z(H) = \gamma_2(H)$ is cyclic, it follows from Theorem \ref{BBC} that $H$ is a central product of $2$-generator groups $H_1, H_2, \ldots, H_m$. Now it is easy to see that $\gamma_2(H_i) = \Z(H_i)$ and $|\gamma_2(H_i)| = q$ for $1 \le i \le m$. This completes the proof of the theorem. \hfill $\Box$

\end{proof}

We are now ready to prove Theorem B.
\vspace{.1in}

\noindent{\it Proof of Theorem B.} Let $G$ be a Camina-type finite $p$-group of nilpotency class $2$. Then by Proposition \ref{prop1b} $G$ satisfies Hypothesis A. That $d(G)$ is even and $d(G) \ge 2d(\gamma_2(G))$, now follows from Theorem  \ref{thmcl2}.  Proposition \ref{cl2prop} tells that $G/\Z(G)$ is homocyclic. The last assertion is Theorem \ref{thmcl2b}. \hfill $\Box$

\section{Camina-type groups with cyclic commutator subgroup}

In this section  we show that  all Camina-type finite $p$-groups of nilpotency class $\ge 3$ with cyclic commutator subgroup are two generator groups. We start with the following interesting result.

\begin{lemma}\label{lemma2a1}
Let $G$ be a  Camina-type finite $p$-group of class $3$ such that $|\gamma_3(G)| = p$. Then the following statements
hold true in $G$:
\begin{subequations} 
\begin{align}
 &\text{$(\gamma_2(G))^p \le \Z(G)$.} \label{cond1} \\
  &\text{$\Z_2(G) = \Phi(G)$,} \label{cond2}
\end{align}
\end{subequations}
where $\Z_2(G)$ denotes the second center of $G$.
\end{lemma}
\begin{proof} 
Suppose that $u \in \gamma_2(G)$ and $x \in G$.  Then  $[u,x] \in \gamma_3(G) \le \Z(G)$, and therefore $[u^p, x] = [u, x]^p = 1$ as $|\gamma_3(G)| = p$. Since this holds true for all $u \in \gamma_2(G)$ and $x \in G$, it follows that $(\gamma_2(G))^p \le \Z(G)$. This proves \eqref{cond1}.

Now we prove \eqref{cond2}. First we show that $\Z_2(G) \le \Phi(G)$. 
Contrarily
assume that $\Z_2(G) \not\le \Phi(G)$. Let $x \in \Z_2(G) - \Phi(G)$.  Since $G$ is Camina-type, 
$[x,G] = \gamma_2(G)$. Since $x \in \Z_2(G)$, it follows that
$\gamma_2(G) = [x,G] \le \Z(G)$. This contradicts the fact that the nilpotency
class of $G$ is $3$. Hence $\Z_2(G) \le \Phi(G)$. Next suppose that $\Z_2(G) <
\Phi(G)$. Since $\gamma_2(G) \le \Z_2(G)$, we can always find an element 
$x \in \Phi(G)-\Z_2(G)$ and an element $y \in G-\Z_2(G)$ such that $x = y^p$. 
For,  if $y^p \in \Z_2(G)$ for all $y \in G-\Z_2(G)$, then $G^p \le \Z_2(G)$ 
and therefore
$\Phi(G) \le \Z_2(G)$ (since  $\gamma_2(G) \le \Z_2(G)$), which we are not
considering. Now for any arbitrary element $v \in G$, we have $[x,v] = [y^p,v] = 
[y,v]^p$  modulo ${\Z(G)}$. But $[y,v]^p \in \Z(G)$ by \eqref{cond1}. Hence 
$[x,v] \in \Z(G)$ for all $v \in G$. This gives that $x \in \Z_2(G)$, which is
  a contradiction to our supposition that $x \in \Phi(G)-\Z_2(G)$. Hence
  $\Z_2(G) = \Phi(G)$, which proves \eqref{cond2} and the proof of the lemma is complete.    \hfill  $\Box$ 

\end{proof}

The following result is also of independent interest.

\begin{lemma}\label{lemma3}
Let $G$ be a finite $p$-group of class $3$ such that $\gamma_2(G)\Z(G)/\Z(G)$
is cyclic. Then $G/\Z_2(G)$ is generated by $2$ elements.
\end{lemma}
\begin{proof}
Let $\{x_1, x_2, \ldots, x_d\}$ be a minimal generating set for $G$. Since $\gamma_2(G)\Z(G)/$ $\Z(G)$ is
cyclic, we can assume without loss of generality that $[x_1,  x_2]\Z(G)$ generates $\gamma_2(G)\Z(G)/\Z(G)$.  Then $[x_1, x_j] = [x_1, x_2]^{\alpha_j}$ modulo $\Z(G)$ for some integer  $\alpha_j$, where $3 \le j \le d$. Since the nilpotency class of  $G/\Z(G)$ is $2$, it follows that $[x_1, x_jx_2^{-\alpha_j}] \in \Z(G)$. Similarly if $[x_2, x_j] = [x_1, x_2]^{\beta_j}$ modulo $\Z(G)$ for some integer  $\beta_j$, where $3 \le j \le d$, then $[x_2, x_1^{\beta_j}x_j] \in \Z(G)$.
Thus it follows that $[x_i, x_1^{\beta_j}x_jx_2^{-\alpha_j}] \in \Z(G)$ for all $i, \; j$ such that $1 \le i \le 2$ and $3 \le j \le d$. Let us set $y_j = x_1^{\beta_j}x_jx_2^{-\alpha_j}$, where $3 \le j \le d$. Notice that the set $\{x_1, x_2, y_3, \ldots, y_d\}$ also generates $G$. 
Thus it follows that $[x_i, y_j] \in \Z(G)$ for all $1 \le i \le 2$ and $3 \le j \le d$. 
Also $[y_j, x_i] \in \Z(G)$ for the same values of $i$ and $j$. Therefore by Hall-Witt identity
\[[x_i, y_j, y_k][y_k, x_i, y_j][y_j, y_k, x_i] = 1\] 
for $1 \le i \le 2$ and $3 \le j, k \le d$, we get 
\begin{equation}\label{equation1}
[y_j, y_k, x_i] = 1.
\end{equation}
Similarly $[x_1, x_2, y_l] = 1$, since $[y_l, x_1, x_2] = 1 = [x_2, y_l, x_1] = 1$, where 
$3 \le l \le d$. Now for all $j, k$ such that $3 \le j, k \le d$, we have
$[y_j, y_k] = [x_1, x_2]^tz$ for some integer $t$ and some $z \in \Z(G)$.
Since $[x_1, x_2, y_l] = 1$, this gives $[y_j, y_k, y_l] = 1$, where $3 \le j,
k, l \le d$. This, along with \eqref{equation1}, proves that 
$[y_j, y_k] \in \Z(G)$ for all $j, k$ such that $3 \le j, k \le d$. 
Since $[y_j, x_i] \in \Z(G)$ for  $3 \le j \le d$ and 
$1 \le i \le 2$, it follows that $[y_j, G] \subseteq \Z(G)$ for all $j$ such that
$3 \le j \le d$. Hence $y_j \in \Z_2(G)$ for $3 \le j \le d$. This proves that $G/\Z_2(G)$ is generated by $x_1\Z_2(G)$ and $x_2\Z_2(G)$, which completes the proof of the lemma.
\hfill $\Box$

\end{proof}

\begin{prop}\label{prop2}
Let $G$ be a Camina-type  finite $p$-group of nilpotency class $3$ such that  $|\gamma_3(G)| = p$ and $\gamma_2(G)\Z(G)/\Z(G)$ is cyclic. Then $G$ is generated by $2$ elements.
\end{prop}
\begin{proof}
Let $G$ be a group satisfying all the conditions of the statement. Then it follows from Lemma \ref{lemma2a1} that $(\gamma_2(G))^p \le \Z(G)$ and $\Z_2(G) = \Phi(G)$. 
Since $\gamma_2(G)\Z(G)/\Z(G)$ is cyclic,  it follows from Lemma \ref{lemma3} that $G/\Z_2(G)$ is generated by $2$ elements and therefore $G/\Phi(G)$ is generated by $2$ elements, since $\Z_2(G) = \Phi(G)$. This shows that $G$ is generated by $2$ elements and the proof of the proposition is complete. \hfill $\Box$

\end{proof}

Here is the main result of this section.

\begin{thm}\label{thm}
If $G$ is a  Camina-type finite $p$-group of nilpotency class at least $3$ with cyclic commutator subgroup, then $d(G) =2$.
\end{thm}
\begin{proof}
 Let $G$ be a finite Camina-type $p$-group of nilpotency class at least $3$. Then there exists a maximal subgroup $N$ of $\gamma_3(G)$ which is normal in $G$.
 Thus it follows from Lemma \ref{lemma2} that $G/N$ is a Camina-type group. Also $G/N$ is of nilpotency class $3$ such that $\gamma_3(G/N)$ is cyclic of order $p$. Thus it follows from Proposition \ref{prop2} that $G/N$ is generated by $2$ elements.  Since $N \le \Phi(G)$, if follows that both  $G$ as well as $G/N$ are minimally generated by equal number of elements. Hence $G$ is generated by $2$ elements and the proof is complete.      \hfill $\Box$

\end{proof}

\section{$2$-generator Camina-type $p$-groups}

In this section we study finite $2$-generator  Camina-type $p$-groups. The following lemma provides a sort of converse of Theorem \ref{thm}.


\begin{lemma}\label{2glemma4}
Let $H$ be a  $2$-generator  Camina-type finite $p$-group of nilpotency class at least $3$ such that $|\gamma_2(H)/\gamma_3(H)| > 2$. Then $\gamma_2(H)$ is cyclic.
\end{lemma}
\noindent{\it Proof.}
Let $H$ be the group as given in the statement and $M$ be a subgroup of $\gamma_3(H)$ of index $p$ which is normal in $H$. Then we can always choose a minimal generating set $\{x, y\}$ of $H$ such that ${\bar x} \in \C_{\bar H}(\gamma_2({\bar H}))$, where ${\bar H} = H/M$ and ${\bar x} = xM$. For,  let $H = \gen{x', y}$ such that  neither ${\bar x'}$ nor ${\bar y}$ centralize $\gamma_2({\bar H})$. Since $\gamma_3({\bar H}) \le \Z({\bar H})$, it gives that neither $[{\bar x'}, {\bar u}]$ nor $[{\bar y}, {\bar u}]$ is the identity element of ${\bar H}$, where $u\gamma_3(H)$ is a generator of the cyclic group $\gamma_2(H)/\gamma_3(H)$. Since both $[{\bar x'}, {\bar u}]$ and  $[{\bar y}, {\bar u}]$ lie in $\gamma_3(\bar{H})$ and $\gamma_3({\bar H})$ is cyclic of order $p$, there exists an integer $t$, $1 \le t < p$, such that $[{\bar y}, {\bar u}]^t = [{\bar x'}, {\bar u}]$. Which implies that $[y^{-t}x', u] = 1$ modulo $M$. Hence $\{x = y^{-t}x', y\}$ is a generating set for $H$ with the required property. Notice that $[x, y] = [y^{-t}x', y] = [x', y]$. For the rest of the proof, we fix such a generating set $\{x, y\}$ for $H$ and assume that $u = [x, y]$. Also let  the orders of $x$ and $y$ modulo $\Z(H)$  be $p^m$ and $p^n$ respectively. For the clarity of the exposition, we complete the proof  in $3$ steps.
\vspace{.1in}

\noindent{\bf Step 1.} {\it If $|\gamma_3(H)| = p$, then $\gamma_2(H)$ is cyclic.}
\begin{proof}
Since $|\gamma_3(H)| = p$, the subgroup $M$ considered above is $1$. So the generating set $\{x, y\}$ of $H$ is such that $x \in \C_H(\gamma_2(H))$ (or equivalently $\gamma_2(H) \le \C_H(x)$).  If $[x, y^s] = 1$ for some integer $s$, then $y^s \in \Z(H)$. Thus it follows that for no positive integer $s < n$,  $[x, y^{p^s}] = 1$, since the order $y$ modulo $\Z(H)$ is $p^n$. Similarly for no positive integer $t < m$,  $[x^{p^{t}}, y] = 1$.   Let $h \in H$ be such that $[x, h] =1$. Then $h = vx^iy^j$ for some  $v \in \gamma_2(H)$ and some integers $i, j$. Since $\gamma_2(H) \le \C_H(x)$, $[x, h] =1$ implies that $[x, y^j] = 1$. Thus $j$ is either zero, or a power of $p$ larger than or equal to $p^n$, and therefore  $\C_H(x) \le \gamma_2(H)\Z(H)\gen{x}$. Obviously  $\C_H(x) \ge \gamma_2(H)\Z(H)\gen{x}$. Hence   $\C_H(x) = \gamma_2(H)\Z(H)\gen{x}$.  

Since $H$ is a Camina-type group,  it follows that $[x, H] = \gamma_2(H)$.  Thus 
\[|\gamma_2(H)| = |[x, H]| = |G : \C_H(x)| = p^n.\]
Since $|\gamma_3(H)| = p$, this implies that  $|\gamma_2(H)/\gamma_3(H)| = p^{n-1}$.   Again, since $[x, H] = \gamma_2(H)$ and $\gamma_2(H) \Z(H)\gen{x} =\C_H(x)$, there exists an element $h \in H - \gamma_2(H)\Z(H) \gen{x}$ of the form $h = y^k$ for some positive integer $k$ such that $\gen{[x, h]} = \gamma_3(H)$  and  $1 \le k < p^n$.

Without loss of generality, we can assume that $k$ is some power of $p$ (by modifying $h$, if necessary). It now follows that $k = p^{n-1}$.  Notice that $n \ge 2$, otherwise $h = y$, which is not  possible as $[x, y]$ generates $\gamma_2(H)$ modulo $\gamma_3(G)$. If  $p$ is odd,  $[x, y^{p^{n-1}}] = [x, y]^{p^{n-1}}$, since $|\gamma_3(H)| = p$.  If $p = 2$, then  our hypothesis  $|\gamma_2(H)/\gamma_3(H)| >2$ tells that $p^n = |\gamma_2(G)| \ge 2^3$, which implies that $n \ge 3$. Hence $[x, y^{p^{n-1}}] = [x, y]^{p^{n-1}}$. Set $w = [x, h] = [x, y^{p^{n-1}}] = [x, y]^{p^{n-1}}$. Since $\gamma_2(H)$ is generated by $\gamma_3(H) = \gen{w}$ and $[x, y]$,  it follows  that $\gamma_2(H)$ is cyclic. 
\end{proof}

\noindent{\bf Step 2.} {\it If  $\gamma_4(H) = 1$, then $\gamma_2(H)$ is cyclic.}
\begin{proof}
First we show that $\gamma_3(H)$ is cyclic. Contrarily assume that $\gamma_3(H)$ is not cyclic. Since, by Lemma \ref{lemma2},  factor group of $H$ by any normal subgroup contained in $\gamma_2(H)$ is Camina-type,  it is sufficient to assume that $\gamma_3(H) \cong \C_p \times \C_p$ (by taking $H$ as $H/(\gamma_3(H))^p$). Notice that $\gamma_3(H)$ is generated by $w_1 := [[x, y], x]$ and $w_2 := [[x, y], y]$.  Since $[x, H] = \gamma_2(H)$, there is an element $h \in H$ such that $w_2 = [x, h]$. Since $\gamma_3(H) \le \Z(H)$, $h$ can  be taken  of the form $y^k[x, y]^l$ for some integers $k$ and $l$. Now
\begin{equation}\label{eqn2g1}
w_2 = [x, h] = [x, y^k[x, y]^l] =  [x, [x, y]^l] [x, y^k] = w_1^{-l} [x, y^k].
\end{equation}
Since the order of $y$ modulo $\Z(H)$ is $p^n$, notice that  $k$ is not congruent to  $0$ modulo $p^n$. Otherwise $w_2$ will be a power of $w_1$, which is not the case. Now \eqref{eqn2g1} tells us that  $1 \neq [x, y^k] \in \gamma_3(H)$. Hence $[x, y^k]^p = [x, y^{pk}]  = 1$, and therefore $y^{pk} \in \Z(H)$. Then $k = k_1 p^{n-1}$ for some non-zero positive integer $k_1$  coprime to $p$.  If $n =1$, then obviously $[x, y^k] = [x, y^{k_1}] = [x, y^{k_1}]^{p^{n-1}}$. So let  $n >1$. Since the exponent of $\gamma_3(H)$ is $p$,  in both of the cases, i.e., $p$ is odd or $p = 2$ and $|\gamma_2(H)/\gamma_3(H)| > 2$, it follows that $[x, y^k] = [x, y^{k_1}]^{p^{n-1}}$. Notice that $[x, y^{k_1}] = [[x, y], y]^{k_1(k_1-1)/2}[x, y]^{k_1} = w_2^{k_1(k_1-1)/2}[x, y]^{k_1}$. Therefore, for $n \ge 2$,  $[x, y^k] = [x, y^{k_1}]^{p^{n-1}} = [x, y]^{k_1p^{n-1}} = [x, y]^k$, since the exponent of $\gamma_3(H)$ is $p$. Thus  \eqref{eqn2g1}  tells us that $w_1^lw_2$ is a power of $[x, y]$. We now claim that $n$ is always $\ge 2$. For,  if $n = 1$, then $k = k_1$. Thus $[x, y^k] = [[x,y], y]^{k(k-1)/2}[x, y]^k$. Since $[x, y^k] \in \gamma_3(H)$, this implies that $[x, y]^k \in \gamma_3(H)$, which is not possible because  $k = k_1$ is coprime to $p$. 

Obviously the set $\{w=w_1^lw_2, w_1\}$ generates $\gamma_3(H)$. Let $N = \gen{w}$. Then $N$ is a normal subgroup of $H$ of order $p$. Set $K = H/N$.  Then, notice that,  $K$ is a Camina-type group and $\gamma_3(K)$ is cyclic of order $p$.  Now we apply  Step 1 to conclude  that modulo $N$, $w_1$ is a power of $[x, y]$. Since $N$ is generated by a power of $[x, y]$, it follows that $w_1$ is a power of $[x, y]$. Thus each element of $\gamma_3(H)$ is a power of $[x, y]$, which is a contradiction. Hence $\gamma_3(H)$ is cyclic. 

Let $T := \gen{w^p}$ be the maximal subgroup of $\gamma_3(H) := \gen{w}$ generated by $w^p$. Thus $T$ is normal in $H$. Consider the quotient group $H/T$, which satisfies all the given hypotheses and $|\gamma_3(H/T)| = p$. Hence by Step 1 $wT$ is a power of $[x, y]T$, and therefore there exist positive integers $\alpha_1, \alpha_2$ such that $wK = [x, y]^{\alpha_1p^{\alpha_2}}T$, where $\alpha_1$ is coprime to $p$. Since $T$ is generated by $w^p$, there exists a non-zero integer  $\beta$ such that $w^{-1}[x, y]^{\alpha_1p^{\alpha_2}} = w^{p \beta}$. Hence $w^{p \beta + 1}$ is a power of $[x, y]$. Since $w^{p \beta + 1}$ also generates $\gamma_3(H)$, and $\gamma_2(H)/\gamma_3(H)$ is generated by $[x, y]\gamma_3(H)$, it follows that $\gamma_2(H)$ is a cyclic subgroup generated by $[x, y]$.  
\end{proof}

\noindent{\bf Step 3.} {\it  $\gamma_2(H)$ is always cyclic.}
\begin{proof}
 We are going to show that $\gamma_2(H)$ is generated by $[x, y]$. We'll do this by proving that $\gamma_2(H)/\gamma_i(H)$ is generated by $[x, y]\gamma_i(H)$ for $4 \le i \le c+1$, where $c$ denotes the nilpotency class of $H$. That $\gamma_2(H)/\gamma_4(H)$ is generated by $[x, y]\gamma_4(H)$ follows from Step 2 as $H/\gamma_4(H)$ is of nilpotency class $3$.  Thus the elements $[x, y, x]$ and $[x, y, y]$ are powers of $[x, y]$ modulo $\gamma_4(H)$. Notice that  $\gamma_4(H)$ is generated by $\gamma_5(H)$ and the elements $[x, y, x, x]$, $[x, y, x, y]$, $[x, y, y, x]$ and $[x, y, y, y]$.  To make the reader little comfortable,  we first show that $\gamma_2(H)/\gamma_5(H)$ is cyclic. It is sufficient to prove that $\gamma_2(S)$ is cyclic, where $S = H/(\gamma_4(H)^p\gamma_5(H))$. We'll now work in $S$ and all elements are considered modulo $\gamma_4(H)^p\gamma_5(H)$. Notice that $\gamma_4(S)$ is elementary abelian group generated by  $[x, y, x, x]$, $[x, y, x, y]$, $[x, y, y, x]$ and  $[x, y, y, y]$. Let $[x,y,x] = [x,y]^{\alpha}v_1$ and  $[x,y,y] = [x,y]^{\beta}v_2$, where $\alpha$ and $\beta$  are non-zero integers which are multiples of $p$, and $v_1, v_2 \in \gamma_4(S)$. Now 
\[[x, y, x, y] = [[x,y]^{\alpha} v_1, y] = [[x,y]^{\alpha}, y] = ([x,y]^{\beta}v_2)^{\alpha} = [x,y]^{\alpha \beta},\]
since $v_1, v_2 \in \Z(S)$ and $\alpha$ is a multiple of $p$.

Similarly, we can show that other generators of $\gamma_4(S)$ are also some powers of $[x, y]$. This proves that $\gamma_2(S)$ is a cyclic subgroup of $S$ generated by $[x,y]$ (modulo  $\gamma_4(H)^p\gamma_5(H)$). Hence  it follows that $\gamma_2(H)/\gamma_5(H)$ is cyclic.

Now, by inductive argument, assume that $\gamma_2(H)/\gamma_{i}(H)$ is a cyclic group generated by $[x, y]\gamma_{i}(H)$. We prove that $\gamma_2(H)/\gamma_{i+1}(H)$ is cyclic.  Consider $S = H/(\gamma_{i}(H)^p\gamma_{i+1}(H))$. Again we read elements modulo $\gamma_{i}(H)^p\gamma_{i+1}(H)$. Notice that $\gamma_i(S)$ is an elementary abelian central subgroup of $S$ generated by at most $2^{i-2}$ commutators of weight $i$. But we know that all commutators of weight $i-1$ are powers of $[x, y]$ modulo $\gamma_{i}(H)$. So a given commutator $c$ of weight $i$ in $S$ can be written as $[[x,y]^{\alpha'}v_1', x]$ or $[[x,y]^{\alpha'}v_1', y]$, where $\alpha'$ is a non-zero integer which is a multiple of $p$, and $v_1' \in \gamma_i(S)$. Also $[x,y,x] = [x,y]^{\alpha''}v_1''$ and $[x,y,y] = [x,y]^{\beta''}v_2''$, where $\alpha''$ and $\beta''$ are non-zero integers which are multiples of $p$, and $v_1'', v_2'' \in \gamma_i(S)$. Since $\gamma_i(S)$ is elementary abelian central subgroup of $S$, it is easy to see that $c$ is some power of $[x, y]$ in $S$. 
Hence it follows that $\gamma_2(H)/\gamma_{i+1}(H)$ is cyclic. Proof   is now complete by inductive argument. \hfill $\Box$

\end{proof}

We are now going to show that most of $2$-generator groups with cyclic commutator subgroup satisfy Hypothesis A, and therefore these are Camina-type groups.
Let $G$ be a $2$-generator finite $p$-group of nilpotency class at least $3$ with cyclic commutator subgroup $\gamma_2(G) = \gen{u}$. Assume that  $|\gamma_2(G)/\gamma_3(G)| > 2$. Notice that the condition $|\gamma_2(G)/\gamma_3(G)| > 2$ is equivalent to the condition $[u, G] \subseteq \gen{u^4}$.  Set $C = \C_G(\gamma_2(G))$. Then $G/C$ is cyclic, because it is isomorphic to the  cyclic group of automorphisms of $\gamma_2(G)$ induced by $G$ by conjugation. Let $y \in G$ be such that $G/C = \gen{yC}$. Then there exists an element $x \in C$ such that $G = \gen{x, y}$ and $[x, u] = 1$.  Let $B = \gen{x, u}$. Then $B$ is an abelian  normal subgroup of $G$ and $G = B\gen{y}$. 
Set $|\gamma_2(G)| = p^r$ for some positive integer $r$. 
Since the order of $\gamma_2(G)$ is $p^r$ and $\gamma_2(G) =  \{[x, y^k] \mid 0 \le k < p^r\} (= \{[x^k, y] \mid 0 \le k < p^r\}$) (by Lemma \ref{yC82}), it follows that $|G/\Z(G)| \ge p^{2r}$. It again follows from Lemma \ref{yC82} that 
$G/\Z(G) \cong \Inn(G) \le \Aut_c(G) = \Inn(G)$. Since, $G$ being minimally generated by two elements,  $|\Aut_c(G)| \le |\gamma_2(G)|^2 = p^{2r}$.  We get  
\[|\Aut_c(G)| = |G/\Z(G)| = p^{2r} .\]
This shows that $G$ satisfies Hypothesis A. Thus we have proved

\begin{lemma}\label{2glemma5}
Let $G$ be a $2$-generator finite $p$-group of nilpotency class at least $3$ such that $\gamma_2(G)$ is cyclic and  $|\gamma_2(G)/\gamma_3(G)| > 2$. Then $G$ satisfies Hypothesis A, and therefore $G$ is Camina-type group.
\end{lemma}

Combining Lemmas \ref{2glemma4} and \ref{2glemma5}, we get
\begin{cor}\label{2gcor1}
If $G$ is a $2$-generator Camina-type finite $p$-group such that $|\gamma_2(G)/\gamma_3(G)| > 2$, then $\gamma_2(G)$ is cyclic and $G$ satisfies Hypothesis A.
\end{cor}

The preceding discussion gives: 

\begin{prop}\label{2gprop1}
Let $G$ be a $2$-generator finite $p$-group of nilpotency class at least $3$ such that $\gamma_2(G)$ is cyclic and  $|\gamma_2(G)/\gamma_3(G)| > 2$. Then $G$ can be  presented by
\begin{equation}\label{2ggrp}
G = \gen{x, y, u \mid x^{p^m} = u^i, y^{p^n} = x^ju^k, [x, y] = u, u^{p^r} = 1 = [u, x], [u, y] = u^{p^s}},
\end{equation}
where $m, n, i, j, k, s$  are some non-negative integers such that  $1 \le s <r$  ($s \ge 2$ if $p = 2$). 
\end{prop}

\begin{thm}\label{2gthm1}
Let $G$ be a $2$-generator  Camina-type finite $p$-group of nilpotency class at least $3$ such  that  $|\gamma_2(G)/\gamma_3(G)| > 2$. Then $G$ is isoclinic to some group defined in \eqref{Intgrp1}.
\end{thm}
\begin{proof}
Let $G$ be a group as in the statement. Then Corollary \ref{2gcor1} tells that $G$ satisfies Hypothesis A and $\gamma_2(G)$ is cyclic.  Set $|\gamma_2(G)| = p^r$.  Let $K$ be a stem group in the isoclinism family of $G$. Then $\Z(K) \le \gamma_2(K)$ and $K$ satisfies Hypothesis A, and therefore $|K/\Z(K)| = |\gamma_2(K)|^2$.  Let us  assume that $K$ is generated by $x$ and $y$. Since $\gamma_2(K)$ is cyclic, we can assume, without loss of any generality,  that  $x$ commutes with the generator $[x, y]$ of $\gamma_2(K)$.   Since $[x, K] = \gamma_2(K)$ and $\gamma_2(K) \le \C_K(x)$, it follows that  no power of $y$ smaller than $p^r$ can land in $\Z(K)$, and therefore no such power of $y$ can lie in $\gamma_2(K)$. Let us see what can lie in $\gamma_2(K)$.  Let $x^iy^j \in \gamma_2(K)$. Since $\gamma_2(K) \le \C_K(x)$, it follows that $x$ commutes with $x^iy^j$. This tells that $[x, y^j] = 1$, and therefore $y^j \in \Z(K)$. Hence $j$ must be either $0$ or some power of $p^r$. So it follows that modulo $\Z(K)$, anybody from $K - K^{p^r}$ that can lie in $\gamma_2(K)$ must be some power of $x$. Suppose that $x^{p^{s_1}}$ is the smallest power of $x$ that  lies in $\gamma_2(K)$. We know for sure that this  $s_1 \le r$, since $\Z(K) \le \gamma_2(K)$. Set $|\Z(K)| = p^t$. Notice that 
\[|K/\Z(K)| = |K/\gamma_2(K)||\gamma_2(K)/\Z(K)| =    p^{r+s_1}p^{r-t} = p^{2r+s_1-t},\]
since $|K/\gamma_2(K)| = p^{r+s_1}$ and $|\gamma_2(K)/\Z(K)| = p^{r-t}$.
But we know that $K/\Z(K)$ is of order $p^{2r}$. Hence the preceding equation implies that  $s_1 = t$. Since the order of $x$ modulo $\Z(K)$ is $p^r$, it follows that $x^{p^t}$ generates $\gamma_2(K)$. This proves that $K$ is metacyclic. Order of $x$ is $p^{r+t}$, $y^{p^r} = x^{p^{r+s}}$ and $[x, y] = x^{p^t}$, for some non-negative integer $s$. If $r =1$, then $|\gamma_2(K)| = p$, and therefore nilpotency class of $K$ is $2$, which we are  not considering. Hence $r \ge 2$. It is now clear that $1 \le t < r$ and $0 \le s \le t$.  Indeed, if $t = r$, then $\Z(K) = \gamma_2(K)$, which we are not considering.  This completes the proof of the  theorem.\hfill $\Box$

\end{proof}

As a consequence of the preceding theorem and  Lemma \ref{2glemma5} we get
\begin{cor}
Let $G$ be a $2$-generator finite $p$-group of nilpotency class at least $3$ such that $\gamma_2(G)$ is cyclic and  $|\gamma_2(G)/\gamma_3(G)| > 2$.  Then $G$ is isoclinic to a metacyclic $p$-group.
\end{cor}

We conclude this section with the following specific non-metacylic $p$-group satisfying Hypothesis A. For an odd prime integer $p$, consider the group
\begin{eqnarray*}
G &=& \langle x, y, u \mid x^{p^2} = y^{p^2} = u^{p^2} = 1, [x, y] = u, [u, x] = u^p, [u, y] = 1\rangle.
\end{eqnarray*}
It is not difficult to show that $G$ is a $2$-generator group of order $p^6$, the nilpotency class of $G$ is $3$,  $\gamma_2(G)$ and $\Z(G)$ are distinct cyclic subgroups of order $p^2$ and $\gamma_3(G)$ is of order $p$.

Since $\gamma_2(G)$ is cyclic, $\Aut_c(G) = \Inn(G)$. Also $|\Inn(G)| = |G/\Z(G)| = p^4 = |\gamma_2(G)|^2$. Hence $G$ satisfies Hypothesis A. But $G$ is not metacyclic.

\section{$2$-generator $2$-groups $G$ with $|\gamma_2(G)/\gamma_3(G)| = 2$}

In this section we classify  finite $2$-generator $2$-groups $G$ satisfying Hypothesis A and with $|\gamma_2(G)/\gamma_3(G)| = 2$. These results are proved jointly with Mike  Newman. It is easy to prove the following two results using computer algebra system Magma \cite{BCP}. For  a given group $G$, let $Y(G)$ denote the normal subgroup $\gamma_4(G)(\gamma_2(G))^2$ of $G$.

\begin{lemma}\label{NYlemma1}
Let $G$ be a finite $2$-generator $2$-group satisfying Hypothesis A and  $|\gamma_2(G)/\gamma_3(G)| = 2$. 
 If the order of $G$ is at most $2^8$, then $Y(G) =  1$.

\end{lemma}

\begin{lemma}\label{NYlemma2}
Let $G$ be a  finite $2$-generator $2$-group such that the nilpotency class of $G$ is at most $3$,  order of $G$ is at most $2^7$ and $\gamma_2(G)$ is elementary abelian of order at most $8$. Then $G$ satisfies Hypothesis A. Furthermore, $\Aut_c(G) = \Inn(G)$ if and only if $|\gamma_2(G)| \le 4$.
\end{lemma}

\begin{thm}\label{NYthm}
Let $G$ be a finite $2$-generator $2$-group such that  $|\gamma_2(G)/\gamma_3(G)| = 2$. Then $G$ satisfies Hypothesis A if and only if the nilpotency class of $G$ is at most $3$ and $\gamma_2(G)$ is elementary abelian of order at most $8$.
\end{thm}
\begin{proof}
Let $G$ be a group as in the statement, which is generated by $x$ and $y$. Suppose that $G$ satisfies Hypothesis A.  To prove the `only if' part of the statement, it is sufficient to show that $Y(G) = 1$. Contrarily assume that  $Y(G) \neq 1$. Then, $G$ being a finite $2$-group, there exists a subgroup $M$ of index $2$ in $Y(G)$ which is normal in $G$. Now consider the quotient group $H = G/M$. Observe that $|Y(H)| = 2$ (for, trivial  $Y(H)$ implies  $Y(G) = M$, which is not true), and the order and exponent of $\gamma_2(H)$ are at most  $16$ and $4$ respectively. Now it follows that both ${\bar x}^4$ and ${\bar y}^4$ lie in $\Z(H)$, where ${\bar x} = x M$ and ${\bar y } = yM$ generate $H$. Let $K$ be a stem group in the isoclinism family of  $H$. Then $K$ is of order at most $2^8$ and $Y(K) \neq 1$. Also $K$ is $2$-generator and satisfies Hypothesis A (by Lemma \ref{prelemma}), and $|\gamma_2(K)/\gamma_3(K)| = 2$. This is not possible by Lemma \ref{NYlemma1}. Hence $Y(K)$ must be trivial, and therefore so is $Y(G)$. 

Conversely, assume that the nilpotency class of $G$ is at most $3$ and $\gamma_2(G)$ is elementary abelian of order at most $8$. Let $K$ be a stem group in the isoclinism family of $G$. Since both $x^4$ and $y^4$ lie in $\Z(G)$, it follows that the order of $K$ is at most $2^7$.  Since $\Z(K) \le \gamma_2(K)$, we have $d(K) = d(G) = 2$.  Then Lemma \ref{NYlemma2} tells that $K$ satisfies Hypothesis A. Since $\Aut_c(G) \cong \Aut_c(K)$ (by Theorem \ref{mY08}), it now follows that $G$ satisfies Hypothesis A. This completes the proof of the theorem. \hfill $\Box$

\end{proof}

The existence  of $2$-generator $2$-groups $G$ satisfying Hypothesis A, having nilpotency classes $2$ and $3$, and having orders $32$, $64$, $128$, $256$ and even higher with $|\gamma_2(G)/\gamma_3(G)| = 2$  can be easily  established using Magma.

Using second assertion of Lemma \ref{NYlemma2}, the following result follows from Theorem \ref{NYthm}.

\begin{prop}\label{NYprop1}
Let $G$ be a finite $2$-generator $2$-group of nilpotency class at least $3$ such that  $|\gamma_2(G)/\gamma_3(G)| = 2$ and $\Aut_c(G) = \Inn(G)$. Then $G$ satisfies Hypothesis A if and only if the nilpotency class of $G$ is  $3$ and $\gamma_2(G)$ is elementary abelian of order $4$.
\end{prop}

\section{Camina-type $p$-groups: Macdonald's argument}

We begin with some standard observations (which can be found in \cite{eK98}) about an arbitrary prime $p$ and an arbitrary finite $p$-group $G$. In this situation there is a natural graded Lie ring

\begin{equation}\label{MAeq1a}
L = L_1 \oplus L_2 \oplus \ldots
\end{equation}
with $i$-th component 
\begin{equation}\label{MAeq1b}
L_i = (\gamma_i(G)/\gamma_{i+1}(G))^+,
\end{equation}
for any integer $i \ge 1$, is the abelian section $\gamma_i(G)/\gamma_{i+1}(G)$ of $G$ rewritten as an additive group instead of multiplicative one. So any element $x_i \in L_i$ is a coset $x\gamma_{i+1}(G)$ for some $x \in \gamma_i(G)$.  Similarly, any $y_j \in L_j$, where $j = 1, 2, \ldots$, is a coset $y \gamma_{j+1}(G)$ for some $y \in \gamma_j(G)$. Then the Lie product $[x_i, x_j]$ in the Lie ring $L$ is the coset
\begin{equation}\label{MAeq1c}
[x_i, x_j] = [x, y]\gamma_{i+j+1}(G) \in L_{i+j}
\end{equation}
of the commutator $[x, y] \in \gamma_{i+j}(G)$ of $x$ and $y$ in $G$. This product is well defined, and turns $L$ into a Lie ring. The usual equation $\gamma_{i+1}(G) = [\gamma_i(G), G]$ for $i = 1, 2, \ldots$ implies that
\begin{equation}\label{MAeq1d}
L_{i+1} = [L_i, L]
\end{equation}
for all integers $i \ge 1$. Hence the additive subgroup $L_1$ generates the Lie ring $L$. Finally, the commutator equation $[x, x] =1$ for all $x \in G$ implies that
\begin{equation}\label{MAeq1e}
[x_i, x_i] = 0
\end{equation}
in $L$ for any $i \ge 1$ and any $x_i \in L_i$. This and the standard Lie identity $[x, y] = -[y, x]$ for all $x, y \in L$ easily imply that $[x, x] = 0$ for all $x \in L$. 

Since each component $L_i$ of $L$ is a finite additive $p$-group, its Frattini subgroup $\Phi(L_i)$ is just $pL_i$. The factor group 
\begin{equation}\label{MAeq2a}
{\bar L}_i := L_i/pL_i = L_i/\Phi(L_i)
\end{equation}
is now a vector space over the finite field ${\mathbb F}_p$ of $p$ elements. After natural identifications, this vector space is the $i$-th component in the graded Lie algebra 
\begin{equation}\label{MAeq2b}
{\bar L} := L/pL = (L_1/pL_1) \oplus  (L_2/pL_2) \oplus \ldots = {\bar L}_1 \oplus {\bar L}_2 \ldots
\end{equation}
over the field ${\mathbb F}_p$.  The equation \eqref{MAeq1d} implies that
\begin{equation}\label{MAeq2c}
{\bar L}_{i+1} = [{\bar L}_i, {\bar L}_1]
\end{equation}
for all $i \ge 1$. Hence the subspace ${\bar L}_1$ generates the Lie algebra ${\bar L}$ over ${\mathbb F}_p$. Finally, the equation \eqref{MAeq1e} implies that
\begin{equation}\label{MAeq2d}
[{\bar x}_i, {\bar x}_i] = 0
\end{equation}
in ${\bar L}$ for all $x_i \in {\bar L}_i$, where $i \ge 1$.

For a vector space $V$ over the finite field ${\mathbb F}_p$, $\text{dim}(V)$ denotes the dimension of $V$. For the rest of this section, for the Lie algebra ${\bar L}$, we assume that 
\begin{eqnarray}
\label{MAeq3a} \text{dim}({\bar L}_3) & = & 1 \quad \text{ and}\\
\label{MAeq3b} [{\bar L}_1, {\bar y}] & = & {\bar L}_2 \quad \text{for all } {\bar y} \in {\bar L}_1 - \{0\}.
\end{eqnarray} 
Following the notation of Macdonald in \cite{iM81}, we denote by $m$ and $n$ the dimensions
\begin{eqnarray}
\label{MAeq4a} \text{dim}( {\bar L}_1) & = & m \quad \text{and}\\
\label{MAeq4b} \text{dim}( {\bar L}_2) & = & n
\end{eqnarray}
of ${\bar L}_1$ and ${\bar L}_2$ respectively, over ${\mathbb F}_p$. It follows from \eqref{MAeq2c} and \eqref{MAeq3a} that $[{\bar L}_2, {\bar L}_1] = {\bar L}_3$ is a non-zero vector space over ${\mathbb F}_p$. Hence neither ${\bar L}_1$ nor  ${\bar L}_2$ can be zero and therefore their dimensions $m$ and $n$ satisfy
\begin{equation}\label{MAeq4c}
m, n > 0.
\end{equation}

An $m \times m$ square matrix $A = (a_{i,j})$ over $\mathbb{F}_p$ is called \textit{strongly skew-symmetric} if $a_{i,j} = - a_{j,i}$ whenever $i \neq j$ and $a_{i,i} = 0$ for all $i$. Notice that when $p$ is odd, a matrix is strongly skew-symmetric if and only if it is skew-symmetric. But when $p = 2$, a matrix is skew-symmetric if and only if it is symmetric, which is not true for strongly skew-symmetric matrices. Thus the notion of strongly skew-symmetric matrix is different from a skew-symmetric one when $p = 2$. Important thing here is that Pfaffians exist only for strongly skew-symmetric matrices, which is a crucial ingredient in the following proof. 
   
Now we are going to reproduce  Macdonald's arguments \cite[pages 353-355]{iM81} in additive setup to prove the following.
\begin{lemma}\label{MAlemma0} 
If $m$ and $n$  are as above, then $m$ is even and $m \ge 2n$.
\end{lemma}
\begin{proof}
Let $\bar{L}_1 \times \bar{L}_1 \rightarrow \bar{L}_2$ be the Lie product function. Let $\bar{x}_1, \bar{x}_2, \ldots, \bar{x}_m$ and  $\bar{y}_1, \bar{y}_2, \ldots, \bar{y}_n$ be the bases of $\bar{L}_1$ and $\bar{L}_2$ respectively. Then 
\begin{equation}\label{L0eq1}
[\bar{x}_i, \bar{x}_j] = {\alpha_{i,j,1}} \bar{y}_1+ \cdots  + {\alpha_{i,j,m}}\bar{y}_n
\end{equation}
for $1 \le i, j \le m$, where $\alpha_{i,j,k} \in \mathbb{F}_p$ for $1 \le k \le n$. Let $A_k$ be the matrix $(\alpha_{j,i,k})$. 
It follows from the defintion of Lie product  and \eqref{MAeq2d} that $A_k$ is a strongly skew-symmetric matrix. 

 Let $\bar{x} = {\phi_1}\bar{x}_1 + \phi_2 \bar{x}_2 + \cdots + \phi_m \bar{x}_m$ be an arbitrary non-trivial element of $\bar{L}_1$, where $\phi_i \in \mathbb{F}_p$ for $1 \le i \le m$. Notice that all of $\phi_i$ can not be zero as $\bar{x}$ is non-trivial element of $\bar{L}_1$.  Since Lie product is ${\mathbb F}_p$-bilinear,  we have
\begin{eqnarray*}
[\bar{x}, \bar{x}_i] &=& \phi_1[\bar{x}_1, \bar{x}_i] + \cdots + \phi_m [\bar{x}_m, \bar{x}_i]\\
&=& \phi_1(\alpha_{1,i,1}\bar{y}_1 + \cdots + \alpha_{1,i,m} \bar{y}_n) + \cdots + \phi_m(\alpha_{m,i,1} \bar{y}_1 + \cdots + \alpha_{m,i,m} \bar{y}_n)\\
&=& \beta_{i,1} \bar{y}_1 + \cdots + \beta_{i,n}\bar{y}_n
\end{eqnarray*}
where
\[\beta_{i,j} = \alpha_{1,i,j}\phi_1 + \cdots + \alpha_{m,i,j}\phi_m.\]
Let $B$ denote the matrix $(\beta_{i,j})$ for $1 \le i \le m$ and $1 \le j \le n$. Notice that $B$ is an $m \times n$ matrix and $B = (A_1\Phi, \ldots, A_n\Phi)$, where $\Phi$ is a column matrix whose transpose $\Phi^T$ is the row matrix $(\phi_1, \phi_2, \ldots, \phi_m)$.

For any non-trivial element $\bar{x} \in \bar{L}_1$, we have a ${\mathbb F}_p$-linear map $[\cdot, \bar{x}] : \bar{L}_1 \times \bar{L}_1 \rightarrow \bar{L}_2$ which maps any $\bar{u} \in \bar{L}_1$ to $[\bar{u}, \bar{x}]$. Now  \eqref{MAeq3b} tells that this map is surjective. Hence $m =\text{dim}(\bar{L}_1) >  \text{dim}(\bar{L}_2) = n$, because the kernel of $[\cdot, \bar{x}]$ is non-trivial. Also, since $[\cdot, \bar{x}]$ is surjective and $\{\bar{x}_1, \bar{x}_2, \ldots, \bar{x}_m\}$ is a basis for $\bar{L}_1$, it follows that the set of images $[\bar{x}_i, \bar{x}]$ of the elements $\bar{x}_i$ for  $1 \le i \le m$ contains a linearly independent subset of cardinality $n$. Thus it follows that $\text{rank}(B)$, the rank of the matrix $B$, is equal to $n$ (this happens because the image $[{\bar x}, \bar{x}_i]$ constitutes $i$th row of the matrix $B$). Hence the matrix $B$ has $n$ linearly independent columns. But there are only $n$ columns in $B$. Therefore all the columns $A_1\Phi, A_2\Phi, \ldots, A_n\Phi$ are linearly independent.   

As mentioned above, each $A_i$ is skew-symmetric. Therefore $A := \psi_1A_i + \psi_1A_2 + \cdots + \psi_nA_n$ is also skew-symmetric for any choice of $\psi_1, \psi_2, \ldots, \psi_n$ in $\mathbb{F}_p$. Suppose that not all of the $\psi$'s are zero. If the determinant of   $A$ is zero, then the determinant of $A\Phi = \psi_1A_1\Phi + \psi_2A_2\Phi + \cdots + \psi_nA_n\Phi$ is also zero. This implies that the determinant of $B$ is zero. But this is a contradiction to the fact that  $A_1\Phi, A_2\Phi, \ldots, A_n\Phi$ are linearly independent. Hence it follows that if  not all of the $\psi$'s are zero, then $A$ is non-singular. As the degree of the square matrix $A$ is $m$ and all skew-symmetric matrices of odd degree are singular, it follows that $m$ is even. Set $m = 2d$.

Now consider the equation
\[|A| = 0,\]
where $|A|$ denotes the determinant of the matrix $A$. The left hand side of the equation is a polynomial $f := f(\psi_1, \psi_2, \ldots, \psi_n)$ of total degree $m = 2d$ in $n$ unknowns $\psi_1, \psi_2, \ldots, \psi_n$ in $\mathbb{F}_p$. Since $A$ is skew-symmetric, there exists a Pfaffian $g := g(\psi_1, \psi_2, \ldots, \psi_n)$ of $A$. Thus $g^2 = f$, and  $f$ is homogeneous. So we get a homogeneous polynomial equation $f = 0$ of total degree $d$ in $n$ variables.  

We claim that $n \le d$. For, if $n > d$, then it  follows from Chevalley-Warning theorem \cite{cC35} that the number of solution of $f = 0$ in $\mathbb{F}_p$ is divisible by $p$.  Since $g$ is homogeneous, it has $\bf{0}$ as a solution, and therefore $g$ has a non-zero solution. This contradicts the fact that $A$ is non-singular. Hence our claim is true. This implies that $m = 2d \ge 2n$, which completes the proof of the lemma. \hfill $\Box$.

\end{proof}

Define ${\bar C}$ to be the ${\mathbb F}_p$-subspace 
\begin{equation}\label{MAeq6}
{\bar C} := \{{\bar y} \in {\bar L}_1 \mid [{\bar L}_2, {\bar y}] = 0\}
\end{equation}
in ${\bar L}_1$. 

\begin{lemma}\label{MAlemma1}
If ${\bar x} \in {\bar L}_1 - {\bar C}$ and ${\bar y} \in {\bar C} - \{0\}$, then $[{\bar x}, {\bar y}] \neq 0$ in the Lie algebra ${\bar L}$.
\end{lemma}
\begin{proof}
Let ${\bar z}$ be an element of ${\bar L}_1$. Since the nilpotency class of our Lie algebra ${\bar L}$ is three,  the Jacobi identity gives
\[[[{\bar x}, {\bar y}], {\bar z}][[{\bar z}, {\bar x}], {\bar y}][[{\bar y}, {\bar z}], {\bar x}] =0.\]
Here $[{\bar z}, {\bar x}] \in [{\bar L}_1, {\bar L}_1] = {\bar L}_2$ and ${\bar y} \in {\bar C}$. Thus by \eqref{MAeq6} $[[{\bar z}, {\bar x}], {\bar y}] = 0$. Hence the above equation is equavalent to 
\[[[{\bar x}, {\bar y}], {\bar z}] = -[[{\bar y}, {\bar z}], {\bar x}].\]

Because ${\bar x} \in {\bar L}_1$ does not lie in ${\bar C}$, by the definition of ${\bar C}$ there exists an element ${\bar w} \in {\bar L}_2$ such that $[{\bar w}, {\bar x}] \neq 0$. Our hypothesis \eqref{MAeq3b} implies that $[{\bar y}, {\bar L}_1] = - [{\bar L}_1, {\bar y}] = {\bar L}_2$ for  the non-zero element ${\bar y}$ of ${\bar L}_1$. Hence we can choose ${\bar z} \in {\bar L}_1$ so that  $[{\bar y}, {\bar z}] = {\bar w}$. Then the above equation becomes
\[[[{\bar x}, {\bar y}], {\bar z}] = -[{\bar w}, {\bar x}] \neq 0.\]
Hence $[{\bar x}, {\bar y}] \neq 0$ and the proof is complete. \hfill $\Box$

\end{proof}

Notice that the ${\mathbb F}_p$-linear maps from ${\bar L}_2$ to ${\bar L}_3$ naturally form a vector space $\Hom({\bar L}_2, {\bar L}_3)$ over ${\mathbb F}_p$. It follows from \eqref{MAeq3a} and \eqref{MAeq4b} that this vector space has dimension
\begin{equation}\label{MAeq8a}
\text{dim}(\Hom({\bar L}_2, {\bar L}_3)) = \text{dim}({\bar L}_2) \cdot \text{dim}( {\bar L}_3) = n \cdot 1 = n.
\end{equation}

Lie multiplication in ${\bar L}$ is ${\mathbb F}_p$-bilinear from ${\bar L}_2 \times {\bar L}_1$ to ${\bar L}_3$. Hence it induces a ${\mathbb F}_p$-linear map $\lambda : {\bar L}_1 \rightarrow \Hom({\bar L}_2, {\bar L}_3)$, sending any ${\bar y}$ in ${\bar L}_1$ to a linear map $\lambda({\bar y}) \in \Hom({\bar L}_2, {\bar L}_3)$ defined by
\begin{equation}\label{MAeq8b}
\lambda({\bar y}) : {\bar x} \mapsto [{\bar x}, {\bar y}] \in {\bar L}_3
\end{equation}
for any ${\bar x} \in {\bar L}_2$. By the definition \eqref{MAeq6} of ${\bar C}$ it follows that ${\bar C}$ is precisely the kernel 
\begin{equation}\label{MAeq8c}
\text{ker}(\lambda) = {\bar C}
\end{equation}
of $\lambda$. Hence we have
\begin{equation}\label{MAeq8d}
m - \text{dim}({\bar C}) = \text{dim}({\bar L}_1/{\bar C}) \le \text{dim}(\Hom({\bar L}_2, {\bar L}_3)) = n,
\end{equation}
where equality holds if and only if $\lambda$ sends ${\bar L}_1$ onto $\Hom({\bar L}_2, {\bar L}_3)$.

\begin{lemma}\label{MAlemma2}
The subspace ${\bar C}$ is a non-zero proper subspace of the vector space ${\bar L}_1$.
\end{lemma}
\begin{proof}
By the definition \eqref{MAeq6} of ${\bar C}$, it is a subspace of ${\bar L}_1$. If ${\bar C} = 0$, then $\text{dim}({\bar C}) = 0$. Thus by Lemma \ref{MAlemma0} and \eqref{MAeq8d} we get $2n \le m \le n$, which is impossible because $n > 0$ by \eqref{MAeq4c}. Hence ${\bar C}$ is a non-zero subspace of ${\bar L}_1$.

If ${\bar C} = {\bar L}_1$, then \eqref{MAeq2c} and \eqref{MAeq6} imply that 
\[{\bar L}_3 = [{\bar L}_2, {\bar L}_1] =  [{\bar L}_2, {\bar C}] = 0.\]
But this contradicts \eqref{MAeq3a}. Hence ${\bar C}$ is a proper subspace of ${\bar L}_1$, and the proof is complete. \hfill $\Box$

\end{proof}

Notice that any element ${\bar z} \in {\bar L}_1$ determines an ${\mathbb F}_p$-linear map $[\cdot, {\bar z}] : {\bar x} \mapsto [{\bar x}, {\bar z}]$ of ${\bar L}_1$ into  ${\bar L}_2$. The kernel of this map is the centralizer 
\begin{equation}\label{MAeq10a}
\C({\bar z}) = \C_{{\bar L}_1}({\bar z}) := \{{\bar x} \in {\bar L}_1 \mid  [{\bar x}, {\bar z}] = 0\}
\end{equation}
of  ${\bar z}$ in ${\bar L}_1$. If ${\bar z} \neq 0$, then \eqref{MAeq3b} implies that $[\cdot, {\bar z}]$ is an epimorphism of the vector space ${\bar L}_1$ onto  ${\bar L}_2$. Thus it follows that
\begin{equation}\label{MAeq10b}
\text{dim}(\C({\bar z})) = \text{dim}({\bar L}_1) -  \text{dim}({\bar L}_2) = m - n
\end{equation}
for all ${\bar z}$ in ${\bar L}_1 - \{0\}$. 

\begin{lemma}\label{MAlemma3}
If ${\bar z} \in {\bar C} - \{0\}$, then $\C({\bar z}) \subseteq {\bar C}$. If ${\bar z} \in {\bar L}_1 - {\bar C}$, then $\C({\bar z}) \cap  {\bar C} = \{0\}$.
\end{lemma}
\begin{proof}
Suppose that ${\bar z}$ in ${\bar C} - \{0\}$, and ${\bar x} \in \C({\bar z}) - {\bar C}$. Then $[{\bar x}, {\bar z}] = 0$ by \eqref{MAeq10a}. But Lemma \ref{MAlemma1} for ${\bar y} = {\bar z}$ implies that $[{\bar x}, {\bar z}] \neq 0$. This contradiction proves the first statement in the lemma.

Now assume that ${\bar z}$ in ${\bar L}_1 - {\bar C}$, and that ${\bar y} \in \C({\bar z}) \cap {\bar C}$ is not $0$. Then by \eqref{MAeq10a} $[{\bar z}, {\bar y}] = - [{\bar y}, {\bar z}] = 0$. But again Lemma \ref{MAlemma1} for ${\bar x} = {\bar z}$ implies that $[{\bar z}, {\bar y}] \neq 0$. Thus we again get a contradiction, and therefore $\C({\bar z}) \cap  {\bar C} = \{0\}$. This completes the proof. \hfill $\Box$

\end{proof}

\begin{prop}\label{MAprop1}
If the Lie algebra ${\bar L}$ satisfies \eqref{MAeq3a} and \eqref{MAeq3b}, then

{\rm (a)} $\quad m = 2n.$

{\rm (b)} $\quad$ If   ${\bar z} \in {\bar L}_1 - \{0\}$, then the subspace  $\C({\bar z})$ of ${\bar L}_1$ has dimesion $n$.

{\rm (c)} $\quad$ The subspace ${\bar C}$ of ${\bar L}_1$ has dimension $n$, and is equal to $\C({\bar y})$ for any ${\bar y} \in {\bar C} - \{0\}$. Hence $\lambda$ is an epimorphism of ${\bar L}_1$ onto $\Hom({\bar L}_2, {\bar L}_3)$ with ${\bar C}$ as its kernel.

{\rm (d)} $\quad$ If ${\bar x} \in {\bar L}_1 - {\bar C}$, then ${\bar L}_1$ is the direct sum $\C({\bar x}) \oplus {\bar C}$ of its subspaces $\C({\bar x})$ and ${\bar C}$. Hence the map ${\bar z} \mapsto [{\bar x}, \bar{z}]$ is an isomorphism of $\bar{C}$ onto $\bar{L}_2$ as $\mathbb F_p$-spaces.
\end{prop}
\begin{proof}
In view of Lemma \ref{MAlemma2} we can choose some elements $\bar{x} \in \bar{L}_1 -\bar{C}$ and $\bar{y} \in \bar{C} - \{0\}$. Then it follows from Lemma \ref{MAlemma3} that  the subspaces $\C(\bar{x})$ and $\C(\bar{y})$ satisfy $\C(\bar{x}) \cap \bar{C} = \{0\}$ and $\C(\bar{y}) \subseteq {\bar C}$. Hence $\C(\bar{x}) \cap \C(\bar{y}) = \{0\}$. It follows that the direct sum $\C(\bar{x}) \oplus \C(\bar{y})$ of these two subspaces is a subspace of $\bar{L}_1$.  In view of \eqref{MAeq4a} and \eqref{MAeq10b}, this implies that 
\[2(m-n) = \text{dim}(\C(\bar{x}) \oplus \C(\bar{y})) \le \text{dim}(\bar{L}_1) = m,\]
where equality holds if and only if  $\bar{L}_1  = \C(\bar{x}) \oplus \C(\bar{y})$. But $m-n \ge n$ by Lemma \ref{MAlemma0}. Therefore $2(m - n) \geq n + (m - n) = m$. So the above inequality is an equality, and $\bar{L}_1$ is the direct sum
\begin{equation}\label{leq1}
\bar{L}_1  = \C(\bar{x}) \oplus \C(\bar{y})
\end{equation}
of the subspaces $\C(\bar{x})$ and $\C(\bar{y})$. Part (a) of the proposition is nothing but  $2(m - n)  = m$.  Part (a) and \eqref{MAeq10b} imply part (b).

Because $\C(\bar{x}) \cap \bar{C} = \{0\}$, the sum $\C(\bar{x}) + \bar{C}$ of subspaces of $\bar{L}_1$ is direct. But $\bar{L}_1$ is already the direct sum \eqref{leq1}, in which the summand $\C(\bar{y})$ is contained in $\bar{C}$. This is possible only when $\C(\bar{y}) = \bar{C}$. Hence $\text{dim}(\bar{C}) = \text{dim}(\C(\bar{y})) = n$, and $\bar{C} = \C(\bar{y})$ for all $\bar{y} \in \bar{C} - \{0\}$. It follows that $m - \text{dim}(\bar{C}) = m - n = n$. Thus equality holds in \eqref{MAeq8d}. This implies that $\lambda$ sends $\bar{L}_1$ onto $\Hom(\bar{L}_2, \bar{L}_3)$. Since $\text{ker}(\lambda) = \bar{C}$ by \eqref{MAeq8c}, this completes the proof of part $(c)$ in the proposition.

The decomposition \eqref{leq1} and the equality $\bar{C} = \C(\bar{y})$ imply that $\bar{L}_1$ is the direct sum $\C(\bar{x}) \oplus \bar{C}$ for any $\bar{x} \in \bar{L}_1 - \bar{C}$. It follows from \eqref{MAeq3b} and \eqref{MAeq10a} that $\C(\bar{x})$ is the kernel of the epimorphism $\bar{z} \mapsto [\bar{x}, \bar{z}]$ of $\bar{L}_1$ onto $\bar{L}_2$. Hence this epimorphism restricts to an isomorphism of $\bar{C}$ onto $\bar{L}_2$. The part (d) holds, and the proposition is proved. \hfill $\Box$

\end{proof}

\section{Proof of Theorems C and  D}

Let us recall that for an arbitrary finite $p$-group $G$, for a prime $p$, $L$ denotes the graded Lie ring defined in \eqref{MAeq1a}, and $\bar{L}$ denotes its factor Lie algebra defined in \eqref{MAeq2b}.  For any integer $i \ge 1$, the additive group $L_i$ is just the multiplicative group $\gamma_i(G)/\gamma_{i+1}(G)$ written additively. Hence $\Phi(L_i) = pL_i$ is the subgroup 
\[\Phi( \gamma_i(G)/\gamma_{i+1}(G)) = \Phi( \gamma_i(G))\gamma_{i+1}(G)/\gamma_{i+1}(G) \le \gamma_i(G)/\gamma_{i+1}(G)\]
written additively. It follows that there is a natural epimorphism $\eta_i$ of the multiplicative group $\gamma_i(G)$ on to the additive group $\bar{L}_i = L_i/\Phi(L_i)$, sending any $x \in \gamma_i(G)$ to
\begin{equation}\label{Ceqn1a}
\eta_i(x) := (x\gamma_{i+1}(G))\Phi(L_i) \in L_i/\Phi(L_i) = \bar{L}_i.
\end{equation}
Furthermore, this epimorphism has the kernel
\begin{equation}\label{Ceqn1b}
\text{ker}(\eta_i) = \Phi( \gamma_i(G))\gamma_{i+1}(G).
\end{equation}
Since Lie multiplication in $\bar{L} = L/pL$ is induced by that in $L$, the definition \eqref{MAeq1c} of the latter multiplication implies that
\begin{equation}\label{Ceqn1c}
[\eta_i(x), \eta_j(y)] = \eta_{i+j}([x, y])
\end{equation}
in $\bar{L}$ for any integers $i,  j \ge 1$, and any elements $x \in \gamma_i(G)$ and $y \in \gamma_j(G)$. Here the commutator  $[x, y]$ is computed in the group $G$, and lies in $[\gamma_i(G), \gamma_j(G)] \le \gamma_{i+j}(G)$. So the right hand side of this equation is well defined.

The situation when $i =1$ is a bit special, since $\gamma_i(G)$ is $G$ itself, and $\gamma_2(G)$ is a subgroup of $\Phi(G) = \Phi(\gamma_1(G))$. It follows that the kenel $\Phi(\gamma_1(G))\gamma_2(G)$ of $\eta_1$ is just $\Phi(G)$. Hence we have

\begin{prop}\label{Cprop1}
The epimorphism $\eta_1$ in \eqref{Ceqn1a} sends the multiplicative group $G$ onto the additive group $\bar{L}_1$, and has $\Phi(G)$ as its kernel. So elements $x_1, x_2, \ldots, x_d \in G$ form a minimal generating set for the $p$-group $G$ if and only if their images $\eta_1(x_1), \eta_1(x_2), \ldots, \eta_1(x_d)$ form a basis of the vector space $\bar{L}_1$ over $\mathbb F_p$. Hence the number $d = d(G)$ of elements in any minimal generating set for $G$ is the dimension $\text{dim}(\bar{L}_1)$ of the $\mathbb F_p$-space $\bar{L}_1$.
\end{prop}

Since $L_1$ is $G/\gamma_2(G)$ written additively, and since $L_1$ generates the Lie ring $L$, an automorphism $\alpha \in \Aut(G)$ lies in $\Aut^{\gamma_2(G)}(G)$, the group of automorphisms  of $G$ which induce the identity on $G/\gamma_2(G)$,  if and only if it induces the identity automorphism on $L$. In particular, any such $\alpha$ induces the identity automorphism on the section $\gamma_i(G)/\gamma_{i+1}(G)$, which is $L_i$ written multiplicatively, for any integer $i \ge 1$.

Suppose that $\{x_1, x_2, \ldots, x_d\}$ is a minimal generating set for $G$, and that $\alpha \in \Aut(G)$. Then there are unique elements $y_1, y_2, \ldots, y_d \in G$ such that $\alpha(x_i) = x_iy_i$ for $1 \le i \le d$. Evidently these elements $y_i$ determine $\alpha$ completely. Furthermore, $\alpha$ lies in the subgroup $\Aut^{\gamma_2(G)}(G)$ of $\Aut(G)$ if and only if $y_i \in \gamma_2(G)$ for $1 \le i \le d$. It follows that $|\Aut^{\gamma_2(G)}(G)| \le |\gamma_2(G)|^d$, with equality if and only if there is some automorphism of the groups $G$ sending $x_i$ to $x_i y_i$ for any choice of  the elements $y_i$ in $\gamma_2(G)$ for $1 \le i \le d$. Hence we have

\begin{prop}\label{Cprop2}
If $p$ is any prime, then the following properties are equivalent for any finite $p$-group $G$:
\begin{subequations}
\begin{align}
& |\Aut^{\gamma_2(G)}(G)| = |\gamma_2(G)|^d.\\
& \text{If  $\{x_1, x_2, \ldots, x_d\}$ is any minimal generating set for $G$, and if  $y_1, y_2, \ldots, y_d$ be any $d$}\\
&  \text{elements in $\gamma_2(G)$, then there exists an automorphism $\alpha$ of $G$ sending $x_i$ to $x_i y_i$ for}\nonumber\\
&  \text{all $1 \le i \le d$.}\nonumber\\
& \text{There is some minimal generating set $\{x_1, x_2, \ldots, x_d\}$ for $G$ such that, for any}\\ 
& \text{$y_1, y_2, \ldots, y_d \in \gamma_2(G)$, there is some automorphism $\alpha$ of $G$ sending $x_i$ to $x_i y_i$ for} \nonumber\\
& \text{all $1 \le i \le d$.}\nonumber
\end{align}
\end{subequations}
\end{prop} 

For the rest of this section we assume that the finite $p$-group satisfies
\begin{subequations}
\begin{align}
|\gamma_3(G)| & =  p,\label{Ceqn4a}\\
[G, y] & = \gamma_2(G) \quad \text{for all} \quad y \in G-\Phi(G), \;\; \text{and}\label{Ceqn4b}\\
|\Aut^{\gamma_2(G)}(G)| &= |\gamma_2(G)|^{d(G)}.\label{Ceqn4c}
\end{align}
\end{subequations}

As a consequence $G$ also satisfies

\begin{lemma}\label{Clemma1}
The kernel $\Phi(\gamma_3(G))\gamma_4(G)$ of the epimorphism $\eta_3: \gamma_3(G) \onto \bar{L}_3$ is $1$. Hence $G$ has class $3$, and $\eta_3$ is an isomorphism of the multiplicative group $\gamma_3(G)$ onto the additive group $\bar{L}_3$. So $\text{dim}(\bar{L}_3) = 1$.
\end{lemma}
\begin{proof} 
Since  the subgroup $\gamma_3(G)$ is of order $p$ (by \eqref{Ceqn4a}), it follows that $\gamma_3(G) \le \Z(G)$ and $\Phi(\gamma_3(G)) = 1$. Hence $G$ is of class $3$ and by \eqref{Ceqn1b} $\eta_3$ is an isomorphism. \hfill $\Box$

\end{proof}

The last statement in the preceding lemma tells us that $\bar{L}$ satisfies \eqref{MAeq3a}. Since $\eta_2$ is an epimorphism of $\gamma_2(G)$ onto $\bar{L}_2$, it follows from \eqref{Ceqn1c} and \eqref{Ceqn4b} that \eqref{MAeq3b} holds. Hence all the hypotheses of Proposition \ref{MAprop1} hold. So from now on we adopt the notation $m, n$, $\bar{C}$, $\C(\bar{z})$, $\lambda$ used in  Proposition \ref{MAprop1}, and apply freely its conclusions (a) through (d).

Following Macdonald, we denote by $C$ the centralizer
\begin{equation}\label{Ceqn6}
C = \C_G(\gamma_2(G))
\end{equation}
of $\gamma_2(G)$ in $G$. It has another description.

\begin{lemma}\label{Clemma2}
The subgroup $C$ of $G$ is the inverse image of the subspace $\bar{C}$ of $\bar{L}_1$ under the epimorphism $\eta_1$ of $G = \gamma_1(G)$ onto $\bar{L}_1$. Hence $C$ contains the kernel $\Phi(G)$ of the epimorphism.
\end{lemma}
\begin{proof}
If $x \in \gamma_2(G)$ and $y \in G$, then \eqref{Ceqn1c} tells us that $\eta_3$ sends $[x, y] \in \gamma_3(G)$ to $[\eta_2(x), \eta_1(y)] \in \bar{L}_3$. Since $\eta_3$ is an isomorphism of $\gamma_3(G)$ onto  $\bar{L}_3$, it follows that $[x, y] =1$ in $G$ if and only if $[\eta_2(x), \eta_1(y)] = 0$  in $\bar{L}_3$.  Since $\eta_2$ is an epimorphism of $\gamma_2(G)$ onto $\bar{L}_2$, we conclude that $y \in G$ satisfies $[x, y] = 1$ for all $x \in \gamma_2(G)$ if and only if its image $\eta_1(y) \in \bar{L}_1$ satisfies $[\bar{x}, \eta_1(y)] = 0$ in $\bar{L}$ for all $\bar{x} \in \bar{L}_2$. In view of \eqref{MAeq6} and \eqref{Ceqn6} this proves the first statement of the lemma. Since $\text{ker}(\eta_1) = \Phi(G)$ by Proposition \ref{Cprop1}, the remaining part of the statement also holds true. \hfill $\Box$

\end{proof}

Now we are ready to prove the key result.

\begin{thm}\label{Cthm1}
If a finite $p$-group, for any prime $p$, satisfies \eqref{Ceqn4a}-\eqref{Ceqn4c}, then $d(G) = 2$.
\end{thm}
\begin{proof}
The number $d(G)$ is equal to the $\text{dim}(\bar{L}_1)$ by Proposition \ref{Cprop1}. This, \eqref{MAeq4a},  \eqref{MAeq4b} and Proposition \ref{MAprop1}(a) imply that $d(G) = m = 2n$. So we only need to show that  $n = 1$. Since the integer $n > 0$ by Lemma \ref{MAlemma0}, it suffices to derive a conradiction from the assumption that $n \ge 2$.

The subspace $\bar{C}$ of $\bar{L}_1$ has dimension $n$ by Proposition \ref{MAprop1}(c). Since $\bar{L}_1$ has dimension $m = 2n$, we can choose elements $\bar{x}_1, \bar{x}_2, \ldots, \bar{x}_n \in \bar{L}_1 - \bar{C}$ forming a basis of $\bar{L}_1$ modulo $\bar{C}$, i.e., whose images form a basis for the factor group $\bar{L}_1/\bar{C}$. Proposition \ref{MAprop1}(c) implies that $\lambda(\bar{x}_1), \lambda(\bar{x}_2), \ldots, \lambda(\bar{x}_n)$ form a $\mathbb F_p$-basis for $\Hom(\bar{L}_2, \bar{L}_3)$.  Let $\bar{w}$ be a basis element of the vector space $\bar{L}_3$ of dimension $1$. Then there is a unique basis $\bar{z}_1, \bar{z}_2, \ldots, \bar{z}_n$ for $\bar{L}_2$ such that $\lambda(\bar{x}_i)$ sends $\bar{z}_j$ to $\delta_{ij}\bar{w}$ for $1 \le i, j \le n$, where $\delta_{ij}$ is the Kronecker delta symbol equal to $1$ when $i = j$ and $0$ otherwise. In view of the definition \eqref{MAeq8b} of $\lambda$, this says that 
\begin{equation}\label{Cleq1}
[\bar{z}_j, \bar{x}_i] = \delta_{ij}\bar{w} \in \bar{L}_3
\end{equation}
for any $1 \le i, j \le n$. Since $\bar{x}_1$ lies in $\bar{L}_1 - \bar{C}$, by Proposition \ref{MAprop1}(d) it follows that the map $\bar{y} \mapsto [\bar{x}_1, {\bar y}]$ is an isomorphism of the vector space  $\bar{C}$ onto $\bar{L}_2$.  Hence there is a unique basis    $\bar{y}_1, \bar{y}_2, \ldots, \bar{y}_n$ for $\bar{C}$ such that 
\begin{equation}\label{Cleq2}
[\bar{x}_1, \bar{y}_j] = \bar{z}_j
\end{equation}
for all $ 1 \le j \le n$.

By definition $\bar{x}_1, \bar{x}_2, \ldots, \bar{x}_n$ form a basis for $\bar{L}_1/\bar{C}$, while $\bar{y}_1, \bar{y}_2, \ldots, \bar{y}_n$ form a baisis for  the subspace $\bar{C}$. Hence  $\bar{x}_1, \bar{x}_2, \ldots, \bar{x}_n, \bar{y}_1, \bar{y}_2, \ldots, \bar{y}_n$ form a basis for $\bar{L}_1$.  In view of Proposition \ref{Cprop1}  there is some minimal family of generators $x_1,  x_2, \ldots, x_n, y_1, y_2, \ldots y_n$ for the $p$-group $G$ such that $\eta_1(x_i) = \bar{x}_i$ and  $\eta_1(y_j) = \bar{y}_j$ for all $1 \le i, j \le n$. Since $[x_1, y_1]$ lies in $\gamma_2(G)$, Proposition \ref{Cprop2} and our assumption \eqref{Ceqn4c} imply that there is some automorphism $\alpha \in \Aut^{\gamma_2(G)}(G)$ sending these generators to 
\begin{equation}\label{Cleq3}
\alpha(x_i) = x_i, \quad \alpha(y_1) = y_1[x_1, y_1], \quad \alpha(y_j) = y_j
\end{equation}
for $1 \le i \le n$ and $2 \le j \le n$.

We claim that the automorphism $\alpha$ fixes  every element of $\gamma_2(G)$. Our assumption that $n \ge 2$ implies that the element $y_2$ exists and is fixed by $\alpha$. Since $\alpha$ also fixes $x_1,  x_2, \ldots, x_n$, it fixes each commutator  $[x_i, y_2] \in \gamma_2(G)$ for $1 \le i \le n$.  We know from \eqref{Ceqn1c} that 
\[\eta_2([x_i, y_2]) = [\eta_1(x_i), \eta_1(y_2)] = [\bar{x}_i, \bar{y}_2]\]
for all $1 \le i \le n$. Because $\bar{y}_2$ lies in $\bar{C}-\{0\}$, its centralizer $\C(\bar{y}_2)$ is $\bar{C}$ by Proposition \ref{MAprop1}(c). So the basis $\bar{x}_1, \bar{x}_2, \ldots, \bar{x}_n$  for $\bar{L}_1$ modulo $\bar{C}$ is also a basis for $\bar{L}_1$ modulo $\C(\bar{y}_2)$. The epimorphism $\bar{x} \mapsto [\bar{x}, \bar{y}_2]$ of $\bar{L}_1$ onto $\bar{L}_2$ has kernel $\C(\bar{y}_2)$. Hence it sends the $\bar{x}_i$, for $1 \le i \le n$, to a basis of $\bar{L}_2$ consisting of all $[\bar{x}_i, \bar{y}_2] = \eta_2([x_i, y_2])$.  We conclude that  $[x_1, y_2], [x_2, y_2], \ldots, [x_n, y_2]$ generate $\gamma_2(G)$ modulo $\gamma_3(G)$. But $\alpha$, like any automorphism in $\Aut^{\gamma_2(G)}(G)$, fixes each element of $\gamma_3(G) = \gamma_3(G)/\gamma_4(G)$. Since it also fixes  each $[x_i, y_2]$ for $1 \le i \le n$, it fixes every elements of $\gamma_2(G)$, and our claim is proved.

Now we consider the action of $\alpha$ on the commutator $[x_1, y_1]$. Since the commutator lies in $\gamma_2(G)$, it must be fixd by $\alpha$. In view of \eqref{Cleq3} this gives
\begin{eqnarray*}
[x_1, y_1] & =& \alpha([x_1, y_1]) = [\alpha(x_1), \alpha(y_1)] =    [x_1, y_1[x_1, y_1]]\\
&=& [x_1, [x_1, y_1]][x_1, y_1]^{[x_1, y_1]} =  [x_1, [x_1, y_1]][x_1, y_1]
\end{eqnarray*}
in $G$. Hence the element $[x_1, [x_1, y_1]] =1$ in $\gamma_3(G)$. Applying $\eta_3$, and using \eqref{Ceqn1c} twice, we obtain
\[[\bar{x}_1, [\bar{x}_1, \bar{y}_1]] = [\eta_1(x_1), [\eta_1(x_1), \eta_1(y_1)]] = \eta_3([x_1, [x_1, y_1]]) = 0\]
in $\bar{L}$.  But
\[ [\bar{x}_1, [\bar{x}_1, \bar{y}_1]] = [\bar{x}_1, \bar{z}_1] = -\bar{w} \neq 0\]
by \eqref{Cleq2} and \eqref{Cleq1}. This contradiction shows that our assumption $n \ge 2$ must be false. Hence $n =1$, which completes the proof of the Theorem. \hfill $\Box$

\end{proof}

Now we are ready to prove Theorems C and D.
\vspace{.1in}

\noindent{\it Proof of Theorem C.}
Let $G$ be a Camina-type finite $p$-group of nilpotency class at least $3$.  Then there exists a maximal subgroup $N$ of $\gamma_3(G)$ which is normal in $G$. 
Set $\bar{G} = G/N$. Notice that $\bar{G}$ is Camina-type (by Lemma \ref{lemma2}), the nilpotency class of $\bar{G}$ is $3$ and $|\gamma_3(\bar{G})| = p$. Now it follows from Proposition \ref{Cprop1} and the preceding discussion for $\bar{G}$ that $d(\bar{G}) = \text{dim}(\bar{L}_1)$ and $d(\gamma_2(\bar{G})/\gamma_3(\bar{G})) = \text{dim}(\bar{L}_2)$.  This, \eqref{MAeq4a},  \eqref{MAeq4b} and Proposition \ref{MAprop1}(a) imply that $d(\bar{G}) = 2d(\gamma_2(\bar{G})/\gamma_3(\bar{G})$.  Since $d(\bar{G}) = d(G)$ and $d(\gamma_2(\bar{G})/\gamma_3(\bar{G})) = d(\gamma_2(G)/\gamma_3(G))$, we have $d(G) = 2d(\gamma_2(G)/\gamma_3(G))$. This proves assertion (i). Assertions (ii) and (iii) follow from Theorem \ref{thm} and Corollary \ref{2gcor1} respectively. \hfill $\Box$

\vspace{.2in}

\noindent{\it Proof of Theorem D.}
Let $G$ be a finite $p$-group of nilpotency class at least $3$ satisfying Hypothesis A. Let $N$ be a maximal subgroup  of $\gamma_3(G)$ which is normal in $G$. Set $\bar{G} = G/N$.
Then $\bar{G}$ satisfies Hypothesis A (by Lemma \ref{lemma1}), the nilpotency class of $\bar{G}$ is $3$ and $|\gamma_3(\bar{G})| = p$. Notice that Hypothesis A is stronger than \eqref{Ceqn4c}, so that any group satisfying Hypothesis A, also satisfies \eqref{Ceqn4c}. So all the hypotheses of Theorem \ref{Cthm1} hold true for $\bar{G}$. Hence $d(\bar{G}) = 2$, and hence $d(G) = 2$. This completes the proof of assertion (i) of the theorem. Assertion (ii) follows from assertion (i),  Lemma \ref{2glemma4}, Lemma \ref{2glemma5}, Proposition \ref{2gprop1} and Theorem \ref{2gthm1}.  Assertion (iii) follows from Theorem \ref{NYthm}.      \hfill $\Box$

\section{Central quotient and commutator subgroup}

Understanding the relationship between $\gamma_2(G)$ and $G/\Z(G)$ goes back, at least, to 1904 when I. Schur \cite{iS04} proved that the finiteness of $G/\Z(G)$ implies the finiteness of $\gamma_2(G)$. A natural question which arises here is about the converse of Schur's theorem, i.e., whether the finiteness of $\gamma_2(G)$ implies the finiteness of $G/\Z(G)$. By a well known result of  P. Hall \cite{pH56} it follows that if  $\gamma_2(G)$ is finite, then $G/\Z_2(G)$ is finite. However, unfortunately, the answer to the converse of Schur's theorem in general is negative as it can be seen for infinite extraspecial $p$-group for an odd prime $p$,  results are available with some extra conditions. We mention here one of these results, which is generalised below.   B. H. Neumann 
\cite[Corollary 5.41]{bN51} proved that  if $G$ is finitely generated and $\gamma_2(G)$ is finite, then $G/\Z(G)$ is finite.  Moreover,  if $G$ is generated by $k$ elements, then $|G/\Z(G)| \le |\gamma_2(G)|^k$. 

We notice that neither of the conditions, i.e., $G$ is finitely generated and $\gamma_2(G)$ finite, is necessary to show that $G/\Z(G)$ is finite. Let $G/\Z(G)$ be finitely generated by a minimal generating set $\{x_1\Z(G), \ldots, x_d\Z(G)\}$ having $d = d(G)$ elements such that $|x_i^{G}|$ is finite for $1 \le i \le d$. Since any class-preserving automorphism of $G$ fixes the center of $G$ elementwise and maps the non-central generating elements to their conjugates, it follows that $|\Aut_c(G)| \le \Pi_{i=1}^d |x_i^G|$. Thus $|G/\Z(G)| = |\Inn(G)| \le |\Aut_c(G)| \le \Pi_{i=1}^d |x_i^G|$ is finite, which implies that $\gamma_2(G)$ is finite. Hence it follows that 
\begin{equation}\label{meq}
|G/\Z(G)| \le |\gamma_2(G)|^d,
\end{equation}
 since $|g^G| = |[g, G]| \le |\gamma_2(G)|$ for all $g \in G$. We have proved

\begin{prop}
Let $G$ be an arbitrary group such that $G/\Z(G)$ is finitely generated by a minimal generating set $\{x_1\Z(G), \ldots, x_d\Z(G)\}$ with $|x_i^G| < \infty$ for $1 \le i \le d$. Then $G/\Z(G)$ is finite. Moreover, $\gamma_2(G)$ is finite and $|G/\Z(G)| \le |\gamma_2(G)|^d$.
\end{prop}

A natural question which arises here is: 
\vspace{.1in}

\noindent {\bf Question.} What are all the groups $G$ for which equality holds in \eqref{meq}?  

\vspace{.1in}

We are going to consider only nilpotent groups $G$ for which equality holds in \eqref{meq}. In this case, there exists a finite nilpotent group $H$ such that  $\Z(H) \le \gamma_2(H)$. Then it follows that $d(H) = d(H/\Z(H)) = d(G/\Z(G)$. So we work with finite nilpotent groups $G$ such that $\Z(G) \le \gamma_2(G)$, and provide an answer to the above question upto isoclinism.  Let $G$ be a finite nilpotent group minimally generated by $d$ elements for which equality holds  in \eqref{meq}. Then it follows that
\[|G/\Z(G)| = |\Inn(G)| \le |\Aut_c(G)| \le  |\gamma_2(G)|^d = G/\Z(G).\]
 Hence $|\Aut_c(G)| = |\gamma_2(G)|^d$, which shows that $G$ satisfies Hypothesis A. Notice that $\Aut_c(G) =  \Inn(G)$. Since $G$ satisfies Hypothesis A, it is a Camina-type group. Thus it follows from Lemma \ref{prelemma2} that $G$ is a $p$-group for some prime integer $p$.  So the problem now  reduces to classifying finite $p$-groups $G$ for which equality holds  in \eqref{meq}. Obvious examples of such groups are finite extraspecial $p$-groups. Other examples are  groups defined in \eqref{Intgrp1} and $2$-generator groups isoclinic to these groups.

For finite $p$-groups of nilpotency class $2$, we get

\begin{thm}\label{s11thm1}
Let $G$ be a finite $p$-group  of nilpotency class $2$ such that  equality holds in  \eqref{meq}.  Then $G$ is isoclinic to the group $Y$ defined in \eqref{ygroup}.
\end{thm}
\begin{proof}
Suppose that $G$ is a group as in the statement.   Since $\Aut_c(G) =  \Inn(G)$ (as noticed above), it follows from Theorem \ref{thmcl2a} that $\gamma_2(G)$ is cyclic. Since $G$ satisfies Hypothesis A, by Theorem \ref{thmcl2b}  $G$ is isoclinic to the group $Y$ defined in \eqref{ygroup}.  \hfill $\Box$

\end{proof}

For finite $p$-groups of nilpotency class larger than $2$, we get

\begin{thm}\label{s11thm2}
Let $G$ be a finite $p$-group  of nilpotency class at least $3$ for which equality holds  in  \eqref{meq}. Then $d(G) = 2$. If  $|\gamma_2(G)/\gamma_3(G)| > 2$, then equality holds for $G$ in  \eqref{meq}  if and only if $G$ is a $2$-generator group with cyclic commutator subgroup.  Furthermore, $G$ is isomorphic to some group defined in \eqref{2ggrp} and  is isoclinic to the group $K$ defined in \eqref{Intgrp1} for suitable parameters. If $|\gamma_2(G)/\gamma_3(G)| = 2$, then equality holds for $G$ in  \eqref{meq}  if and only if $G$ is a $2$-generator $2$-group of nilpotency class $3$ with elementary abelian $\gamma_2(G)$ of order $4$.
\end{thm}
\begin{proof}
 Since  any group $G$, as in the statement, satisfies Hypothesis A, the proof of all assertions, except the last one, follows from Theorem D stated in the introduction.  The last assertion follows from Proposition \ref{NYprop1}. \hfill $\Box$

\end{proof}

\vspace{.2in}

\noindent{\bf Acknowledgements.} I thank Prof. Everett Dade for his timely help and valuable suggestions. Results of Sections 9 and 10 are suggested by him.  I thank Prof. Mike Newman for  his valuable suggestions on the presentation of the paper, for his help in fixing the condition $|\gamma_2(G)/\gamma_3(G)| > 2$ for $2$-groups and for his help in writting a Magma code for computing class-preserving automorphism group of finite groups.  Section 8 is written jointly with him. Thanks are also due to my student Mr. Pradeep K. Rai for reading an earlier version of the paper and pointing out some corrections. This work was completed during  my visit to the mathematics department of the Australian National University, Canberra, where I was visiting under Indo-Australia early career visiting fellowship of Department of Science and Technology, Govt. of India, implemented by Indian National Science Academy.

\end{document}